\documentclass[a4paper]{article}
\usepackage[english]{babel}
\usepackage[utf8x]{inputenc}
\usepackage[T1]{fontenc}
\usepackage[a4paper,top=3cm,bottom=2cm,left=3cm,right=3cm,marginparwidth=1.75cm]{geometry}
\usepackage{empheq}
\usepackage{amsmath, amsthm, amssymb}
\usepackage{graphicx}
\usepackage{multirow}
\usepackage[colorlinks=true, allcolors=blue]{hyperref}

\AtBeginDocument{
\newtheorem{prop}{Proposition}[]

}

\newcommand{\bu}[0]{\mathbf{u}}
\newcommand{\bx}[0]{\mathbf{x}}
\newcommand{\bn}[0]{\mathbf{n}}

\newcommand{\dd}[0]{\mathrm{d}}
\newcommand{\EOS}[0]{\mathrm{EOS}}
\newcommand{\KCFL}{K_\mathrm{CFL}}
\newcommand{\Fr}{\mathrm{Fr}}
\newcommand{\exex}{\mathrm{EXEX}}
\newcommand{\imex}{\mathrm{IMEX}}
\newcommand{\acou}{\mathrm{Acou}}
\newcommand{\trans}{\mathrm{Trans}}

\newcommand{\cO}{\mathcal{O}}
\newcommand{\cN}{\mathcal{N}}

\newcommand{\nbR}{\mathbb{R}}
\newcommand{\nbZ}{\mathbb{Z}}

\newcommand{\ndot}{\nabla \cdot}
\newcommand{\dt}{\Delta t}
\newcommand{\dx}{\Delta x}
\newcommand{\pt}{\partial_t}
\newcommand{\px}{\partial_x}

\title{An all-regime and well-balanced Lagrange-projection type scheme for the shallow water equations on unstructured meshes}
\author{Christophe Chalons\thanks{Laboratoire de Math\'ematiques de Versailles, UMR 8100, Universit\'e de Versailles Saint-Quentin-en-Yvelines, UFR des Sciences, b\^atiment Fermat, 45 avenue des Etats-Unis, 78035 Versailles cedex, France,  (christophe.chalons@uvsq.fr).}
\and{Samuel Kokh\thanks{CEA/DEN/DANS/DM2S/STMF/LMEC, CEA Saclay, 91191 Gif-sur-Yvette, France, (samuel.kokh@cea.fr).}}
\and{Maxime Stauffert\thanks{Laboratoire de Math\'ematiques de Versailles, UMR 8100, Universit\'e de Versailles Saint-Quentin-en-Yvelines, UFR des Sciences, b\^atiment Fermat, 45 avenue des Etats-Unis, 78035 Versailles cedex, France, (maxime.stauffert@uvsq.fr).}\textsuperscript{~}\thanks{Maison de la Simulation, USR 3441, Digiteo Labs, b\^atiment 565, CEA Saclay, 91191 Gif-sur-Yvette, France.}}}
\begin{document}
\maketitle

\begin{abstract}
In this work, we focus on the numerical approximation of the shallow water equations in two space dimensions.
Our aim is to propose a well-balanced, all-regime and positive scheme.
By well-balanced, it is meant that the scheme is able to preserve the so-called lake at rest smooth equilibrium solutions.
By all-regime, we mean that the scheme is able to deal with all flow regimes, including the low-Froude regime which is known to be challenging when using usual Godunov-type finite volume schemes.
At last, the scheme should be positive which means that the water height stays positive for all time. Our approach is based on a Lagrange-projection decomposition which allows to naturally decouple the acoustic and transport terms.
Numerical experiments on unstructured meshes illustrate the good behaviour of the scheme.
\end{abstract}

\section{Introduction}

We are interested in the numerical approximation of the shallow water equations (SWE)
\begin{subequations}
\label{eq: intro_swe}
\begin{empheq}[left=\empheqlbrace]{align}
&\pt h + \ndot (h\bu) = 0,\\
&\pt (h\bu) + \ndot (h\bu\otimes\bu) + \nabla \frac{gh^2}{2} = -gh\nabla z,
\end{empheq}
\end{subequations}
where $\bx\in\nbR^2 \mapsto z(\bx)$ denotes a given smooth topography and $g > 0$ is the gravity constant.
Both the water depth $h$ and the velocity $\bu=(u_1,u_2)\in\nbR^2$ depend on the space and time variables, namely $\bx \in \nbR^2$ and $t \in [0,\infty)$. 
We assume that the initial water depth $h(\bx,t=0)=h_0(\bx)$ and velocity $\bu(\bx,t=0)=\bu_0(\bx)$ are given.

Let us briefly properties of system~\eqref{eq: intro_swe} in the case $\nabla z = 0$~: the system is strictly hyperbolic over the phase space $\Omega=\{(h,h\bu)\in\nbR^3~|~h>0\}$. 
Moreover, if $\bn\in\nbR^2$ is an arbitrary unit vector, the eigenstructure of~\eqref{eq: intro_swe} is composed by two genuinely nonlinear characteristic fields associated with the eigenvalues $\{\bu^T\bn-c,\bu^T\bn+c\}$, where $c \coloneqq \sqrt{gh}$ is the sound speed, and a linearly degenerated field associated with the eigenvalue $\bu^T\bn$. 
We recall also that the regions where $(\bu^T\bn)^2<c^2$ (resp. $(\bu^T\bn)^2>c^2$) are called subcritical or subsonic (resp. supercritical or supersonic).

We are interested in this work in developing a numerical scheme that satisfies the well-balanced property.
More specifically we want our scheme to strictly preserve the "lake at rest" steady solutions, that are the states satisfying
\[
h+z=\mathrm{constant}, \qquad \bu = \mathbf{0}.
\]
For a review on numerical schemes that satisfy the so-called well-balanced property we refer the reader to the pioneering work~\cite{BV94}, books~\cite{Bou04} and \cite{Gos13}.
We also refer to~\cite{CKKS17} where the authors focus on the 1D case and already propose a well-balanced Lagrange-projection strategy.
At last, in~\cite{CCL18} the proposed Lagrange-projection scheme is exact for a full set of equilibrium solutions (and not only the lake at rest).

The Lagrange-projection methodology is especially well suited for subsonic or near low-Froude number flows.
We use an implicit-explicit strategy that allows to keep a stable scheme under a CFL time step limitation which is driven only by (slow) material waves and not by (fast) acoustic waves.
The implicit-explicit Lagrange-projection~\cite{GR96} scheme is designed following the pioneering work~\cite{CNPT10}.
More recent works are concernet with the case of Euler systems in the large friction or low-Mach regimes~\cite{CGK13,CGK14,CGK16} for single or two-phase flow models.
The treatment of the low Froude number is considered through the so called all-regime (or asymptotic-preserving) property and follows the anti-diffusive technique on the pressure numerical flux introduced in~\cite{Del10} and also used in~\cite{CGK16}.

The SWE has been largely studied and one can find nice overviews and references in the books \cite{Bou04} and~\cite{Gos13}.
The scheme proposed in~\cite{CKKS17} in one dimension has been studied in the framework of SWE and more specifically its behaviour for low-Froude number flows in~\cite{Zak17}.
A different implicit-explicit methodology in two dimension context has been proposed by~\cite{BALMN14}.

In section 2, we study the dimensionless system associated to the SWE~\eqref{eq: intro_swe} and its asymptotic limit in low Froude regimes.
In section 3, we present the Lagrange-projection like acoustic / transport decomposition associated to system~\eqref{eq: intro_swe}.
In section 4, we present the schemes, the finite volume scheme in 1D, the study of its truncation error in low Froude regimes and the proposed correction, and finally the extension towards 2D schemes on unstructured meshes.
At last, we show some numerical results in 2D to verify the well-balanced property and illustrate the behaviour of the scheme in different regimes, especially in the low Froude one.
\section{Low Froude limit for continuous equations}

\subsection{Dimensionless shallow water equations}
In this section, we briefly introduce the dimensionless SWE.
These equations will be useful to study the low-Froude asymptotic behaviour of the solutions of~\eqref{eq: intro_swe}.
With this in mind, we define the following dimensionless quantities
\begin{align*}
\tilde{t} = \frac{t}{T}, && \tilde{\bx} = \frac{\bx}{L}, && \tilde{h} = \frac{h}{h_0}, && \tilde{\bu} = \frac{\bu}{u_0}, && \tilde{z} = \frac{z}{z_0},
\end{align*}
where $T$, $L$, $h_0$, $u_0$ and $z_0$ are respectively reference time, length, water height, velocity and topography such that
\begin{equation*}
u_0 = \frac{L}{T} \quad \mathrm{and} \quad z_0 = h_0.
\end{equation*}
Defining the Froude number $\Fr$ by 
$$\Fr = \frac{u_0}{c_0},$$ 
where $c_0 = \sqrt{gh_0}$ is the reference sound speed, easy calculations then give the dimensionless SWE
\begin{subequations}\label{eq: low_froude_swe}
\begin{empheq}[left=\empheqlbrace]{align}
&\partial_{\tilde{t}} \tilde{h} + \nabla_{\tilde{\bx}} \cdot \left(\tilde{h}\tilde{\bu}\right) = 0, 
\\ 
&\partial_{\tilde{t}} \left(\tilde{h}\tilde{\bu}\right) + \nabla_{\tilde{\bx}} \cdot \left(\tilde{h} \tilde{\bu} \otimes \tilde{\bu}\right) + \frac{1}{\Fr^2} \nabla_{\tilde{\bx}} \frac{\tilde{h}^2}{2} = -\frac{1}{\Fr^2} \tilde{h} \nabla_{\tilde{\bx}} \tilde{z}.
\end{empheq}
\end{subequations}

\subsection{Asymptotic equations in low Froude limit}
In this section, we give the asymptotic behaviour of the solutions of the SWE equations in the low Froude limit. 
If we omit the tilde notation for the sake of readability in system~\eqref{eq: low_froude_swe} and, if we introduce the dimensionless pressure function $p(h)=\frac{h^2}{2}$, we get
\begin{subequations}\label{eq: low_froude_swe2}
\begin{empheq}[left=\empheqlbrace]{align}
&\pt h + \ndot \left(h\bu\right) = 0, 
\\ 
&\pt \left(h\bu\right) + \ndot \left(h \bu \otimes \bu\right) + \frac{1}{\Fr^2} \nabla p = -\frac{1}{\Fr^2} h \nabla z.
\end{empheq}
\end{subequations}
Let us assume that $h$ and $z$ admit the following expansions in powers of the Froude number:
\begin{align*}
h = h^{(0)} + h^{(1)} \Fr + h^{(2)} \Fr^2 + \cO(\Fr^3) && \text{and} && \bu = \bu^{(0)}
 + \bu^{(1)} \Fr + \bu^{(2)} \Fr^2 + \cO(\Fr^3),
\end{align*}
which gives in particular 
\[p = p^{(0)} + p^{(1)} \Fr + p^{(2)} \Fr^2 + \cO(\Fr^3) = p(h^{(0)}) + h^{(1)} p^{\,\prime}(h^{(0)}) \Fr + \cO(\Fr^2). \]
The governing equations give at order -2 and -1 with respect to the Froude number that 
\begin{align*}
\nabla p^{(0)} + h^{(0)} \nabla z = 0 &\Leftrightarrow \nabla h^{(0)} = - \nabla z \Leftrightarrow h^{(0)} + z = H(t),
\\
\nabla p^{(1)} + h^{(1)} \nabla z = 0 &\Leftrightarrow h^{(0)} \nabla h^{(1)} = 0 \Leftrightarrow \nabla h^{(1)} = 0 \Leftrightarrow h^{(1)} = h^{(1)}(t).
\end{align*}
The asymptotic behavior is then given by

\begin{subequations}\label{eq: dev first order}
\begin{empheq}[left=\empheqlbrace]{align}
\pt h^{(0)} + \ndot (h^{(0)} \bu^{(0)}) &= 0,
\label{eq: dev first order h}
\\
\pt (h^{(0)} \bu^{(0)}) + \ndot (h^{(0)} \bu^{(0)} \otimes \bu^{(0)}) + \nabla p^{(2)} &= - h^{(2)} \nabla z.
\end{empheq}
\end{subequations}
Now if we impose one of the following velocity boundary conditions
\begin{align*}
\left(\int_{\Omega} \ndot \bu \, \dd\Omega = 0 \text{ and } \int_{\Omega} \ndot (z\bu) \, \dd\Omega = 0\right)  
&& \text{or} && 
\left(\int_{\Omega} \ndot (h \bu) \, \dd\Omega = 0\right),
\end{align*}
integrating \eqref{eq: dev first order h} with respect to the space  variable gives
\begin{align*}
0 &= \int_{\Omega} \left(\pt h^{(0)} + \ndot (h^{(0)} \bu^{(0)})\right) \, \dd\Omega
\\
&= \int_{\Omega} \pt (H - z) \, \dd\Omega + \int_{\Omega} \ndot (\left(H - z\right) \bu^{(0)}) \, \dd\Omega
\\
&= \int_{\Omega} \pt H \, \dd\Omega + H \int_{\Omega} \ndot \bu^{(0)} \, \dd\Omega - \int_{\Omega} \ndot (z \bu^{(0)}) \, \dd\Omega
\\
&=  \pt \int_{\Omega} H \, \dd\Omega = \lvert \Omega \rvert \pt (h^{(0)} + z) = \lvert \Omega \rvert \pt h^{(0)}
\end{align*}
thus $\pt h^{(0)} = 0$ and $h^{(0)} + z = H$ is constant both in space and time.
This leads to $\ndot (h^{(0)}\bu^{(0)}) = 0$ and therefore
\[
\ndot \bu^{(0)} = \ndot (\frac{z}{H}\bu^{(0)}),
\]
while the evolution of $\bu$ is given by
\begin{equation*}
\left(1-\frac{z}{H}\right)\pt \bu^{(0)} + \ndot (\bu^{(0)} \otimes \bu^{(0)}) + \frac{1}{H} \nabla p^{(2)}
= \ndot (\frac{z}{H}\bu^{(0)} \otimes \bu^{(0)}) - h^{(2)} \nabla \frac{z}{H}.
\end{equation*}
Notice that when the topography is flat, \textit{i.e.} $z=0$, the three equations
\begin{equation*}
\left\{
\begin{aligned}
&h^{(0)} + {z} = H = \text{cst} 
\\
&\ndot \bu^{(0)} = \ndot (\frac{z}{H}\bu^{(0)})
\\
&\left(1-\frac{z}{H}\right)\pt \bu^{(0)} + \ndot (\bu^{(0)} \otimes \bu^{(0)}) + \frac{1}{H} \nabla p^{(2)}
= \ndot (\frac{z}{H}\bu^{(0)} \otimes \bu^{(0)}) - h^{(2)} \nabla \frac{z}{H}
\end{aligned}
\right.
\end{equation*}
degenerate towards the incompressible Euler equations
\begin{equation*}
\left\{
\begin{aligned}
&h^{(0)} = \mathrm{cste} 
\\
&\ndot \bu^{(0)} = 0
\\
&\pt \bu^{(0)} + \ndot (\bu^{(0)} \otimes \bu^{(0)})
+ \frac{1}{h^{(0)}} \nabla p^{(2)} = 0.
\end{aligned}
\right.
\end{equation*}

\section{An acoustic/transport operator decomposition}\label{sec: splitting}
Let us first introduce notations related to our discretization.
We suppose that the computational domain $\Omega \subset \nbR^2$ is covered by $N$ polygonal cells $\left(\Omega_j\right)_{1\leq j \leq N}$.
We consider $\Gamma$, a face of the cell $j$, and we suppose the following admissibility assumptions are satisfied:

\begin{itemize}
\item either there exists a single $1\leq k\leq N$ such that $\Gamma=\overline{\Omega_j}\cap\overline{\Omega_k}\not= \emptyset$.
In this case we note  $\Gamma = \Gamma_{jk}$ and $\Gamma_{jk}$ can either be a vertex or a single face of the mesh,
\item either $\Gamma\subset \partial \Omega$ and we suppose that there exists a single $k>N$ that will help to index ghost values for boundary conditions and we shall note $\Gamma = \Gamma_{jk}$.
\end{itemize}
For $1\leq j\leq N$, we note $\cN(j)$ the set of indices $k$ such that $\Gamma_{jk}$ is a face of $\Omega_j$ and if $k\in\cN(j)$ we set $\bn_{jk}$ to be the unit normal vector to $\Gamma_{jk}$ pointing out of $\Omega_j$.

We can now turn to the acoustic~/~transport decomposition of the system~\eqref{eq: intro_swe}.
If we develop the spatial derivatives and isolate the transport terms $(\bu\cdot\nabla)\varphi$, where $\varphi = h, h\bu$, we can use a splitting operator with respect to time to obtain on one hand the acoustic step
\begin{equation}\label{eq: splitting_acoustic}
\pt h + h \nabla\cdot(\bu) = 0,
\quad
\pt (h \bu) + h \bu (\ndot \bu) + \nabla p= -g h \nabla z,
\end{equation}
and on the other hand the transport step
\begin{equation}\label{eq: splitting_transport}
\pt h + (\bu\cdot\nabla) h =0
,\quad
\pt (h\bu )+ (\bu\cdot\nabla)  (h\bu) =0.
\end{equation}

With these notations, the Lagrange-projection algorithm is defined as follows: for a given discrete state $(h,h\bu)^{n}_{j}$, $j\in\nbZ$, defining $(h,h\bu)^{n+1}_{j}$ is a two-step process defined as follows
\begin{enumerate}
\item Update $(h,h\bu)^{n}_{j}$ to $(h,h\bu)^{n+1-}_{j}$ by approximating the solution of system~\eqref{eq: splitting_acoustic},
\item Update $(h,h\bu)^{n+1-}_{j}$ to $(h,h\bu)^{n+1}_{j}$ by approximating the solution of system~\eqref{eq: splitting_transport}.
\end{enumerate} 

{\it Relaxation approximation of the acoustic system.}
Before entering the details of these two steps in the following section, let us note that if we denote $\tau=1/h$, by simple manipulations system~\eqref{eq: splitting_acoustic} can be recast into:
\begin{equation*}
\pt \tau - \tau(\bx,t) \nabla\cdot \bu =0,
\quad
\pt \bu + \tau(\bx,t) \nabla p = - \tau(\bx,t)  \frac{g}{\tau} \nabla z.
\end{equation*}
Following~\cite{CNPT10}, we will choose to approximate the solution of system~\eqref{eq: splitting_acoustic} thanks to a Suliciu-relaxation process.
More precisely we will solve
\begin{equation}\label{eq: splitting_relaxation}
\left\{
\begin{aligned}
\pt \tau &- \tau(\bx,t) \nabla\cdot \bu &&=0,
\\
\pt \bu &+ \tau(\bx,t) \nabla \Pi &&= - \tau(\bx,t)  \frac{g}{\tau} \nabla z,
\\
\pt \Pi &+ \tau(\bx,t) a^2 \nabla\cdot \bu &&= \lambda ( p^{\EOS}(\tau)- \Pi),
\end{aligned}
\right.
\end{equation}
with $p = p^{\EOS}(\tau) = g/(2\tau^2)$, in the regime $\lambda \rightarrow +\infty$.
The parameter $a$ is a constant that is chosen in agreement with the subcharacteristic stability conditions that will be given later.
Over the time interval $\left[t^n,t^n+\dt\right)$, we can account for the limit $\lambda \rightarrow +\infty$ by setting  $\Pi(\bx,t^n) = p^{\EOS}( \tau(\bx,t^n))$, and then solving the relaxed system with $\lambda=0$.
We add another approximation by supposing that over $\left[t^n,t^n+\dt\right)$ it is reasonable to replace $\tau(\bx,t)\partial_{x_r}$ by $\tau(\bx,t^n)\partial_{x_r}$, $r=1,2$.
Finally, we will define our approximation of the acoustic system~\eqref{eq: splitting_acoustic} by solving
\begin{equation}\label{eq: splitting_relaxation2}
\left\{
\begin{aligned}
\pt \tau &- \tau(\bx,t^n) \nabla\cdot \bu &&= 0,
\\
\pt \bu &+ \tau(\bx,t^n) \nabla \Pi  &&= - \tau (\bx,t^n) \frac{g}{\tau} \nabla z,
\\
\pt \Pi &+ \tau(\bx,t^n) a^2 \nabla\cdot \bu &&= 0,
\end{aligned}
\right.
\end{equation}
over $\left[t^n,t^n+\dt\right)$, with $\Pi(\bx,t^n) = p^{\EOS}( \tau(\bx,t^n))$.

Note that the system~\eqref{eq: splitting_relaxation2} is rotational invariant.
This will allow us in the following to define every flux for the two dimensional problem in the reference frame associated to each face.
In this last referential, the problem will be reduced to a quasi-one dimensional problem that we study in the beginning of next section.
One can also notice that the eigenstructure of~\eqref{eq: splitting_relaxation2} in the phase space $\left\{(h,h\bu^T,\Pi,z)\in \nbR^5, h>0, z>0\right\}$ is very simple since it has three eigenvalues $\left\{-a,0,a\right\}$ all associated to linearly degenerated characteristic fields.
\section{Finite volume approximation}
In this paragraph, we present in details the first-order finite volume scheme associated with the acoustic / transport decomposition of section~\ref{sec: splitting}. 

\subsection{A well-balanced Lagrange-projection finite volume scheme in 1D}\label{ssec: scheme_1D}
We start by considering one-dimensional problems and briefly recall the method proposed in \cite{CKKS17}.
In this case the Saint-Venant equations read
\[\left\{
\begin{aligned}
&\pt h + \px (h u_1) = 0,
\\
&\pt (h u_1) + \px \left(h u_1^2 + g\frac{h^2}{2}\right) = -g h\px z,
\\
&\pt (h u_2) + \px (h u_1 u_2) = 0.
\end{aligned}
\right.\]
The system associated with the acoustic step reads
\[\left\{
\begin{aligned}
&\pt \tau - \tau(x,t^n) \px u_1 = 0,
\\
&\pt u_1 + \tau(x,t^n) \px \Pi =  - \tau(x,t^n)\frac{g}{\tau}\px z,
\\
&\pt u_2 = 0,
\\
&\pt \Pi + \tau(x,t^n) a^2 \px u_1 =  0,
\end{aligned}
\right.\]
and the system that accounts for transport boils down to
$$
\pt \varphi + u_1 \px \varphi = 0,
\qquad
\varphi\in\{h, u_1, u_2\}.
$$

We suppose given a strictly increasing sequence $x_{j+1/2} \in\nbR$, for $j\in\nbZ$ and we consider the set of cells $\Omega_j = [x_{j-1/2},x_{j+1/2})$.
The local space step is defined by $\dx_j = x_{j+1/2} - x_{j-1/2}$.
We note $\dt >0 $ the time step and we set $t^n = n\dt$ for $n\in\mathbb{N}$.

The following discretization strategy was presented in~\cite{CKKS17}:  the acoustic step~\eqref{eq: splitting_relaxation2} is approximated by
\begin{subequations}\label{eq: scheme_1D_acoustic}
\begin{empheq}[left=\empheqlbrace]{align}
\tau_j^{n+1-} = \tau_j^n - \tau_j^n \frac{\dt}{\dx_j} \left( u^\sharp_{j+1/2} - u^\sharp_{j-1/2} \right)
\\
(u_1)_j^{n+1-} = (u_1)_j^n - \tau_j^n \frac{\dt}{\dx_j} \left( \Pi^{L,\sharp}_{j+1/2} - \Pi^{R,\sharp}_{j-1/2} \right)
\\
(u_2)_j^{n+1-} = (u_2)_j^n 
\\
\Pi_j^{n+1-} = \Pi_j^n - \tau_j^n \frac{\dt}{\dx_j} a^2 \left( u^\sharp_{j+1/2} - u^\sharp_{j-1/2} \right)
\end{empheq}
\end{subequations}
where for all $j$, $\Pi_j^n = g\frac{{h_j^n}^2}{2}$ and the numerical fluxes $u^\sharp_{j+1/2}$, $\Pi^{L,\sharp}_{j+1/2}$ and $\Pi^{R,\sharp}_{j-1/2}$ are defined by
\begin{align*}\label{eq: scheme_1D_flux}
u^\sharp_{j+1/2} &= u_\Delta (\mathbf{U}_{j}^\sharp,\mathbf{U}_{j}^n,\mathbf{U}_{j+1}^\sharp,\mathbf{U}_{j+1}^n),
\\
\Pi^{R,\sharp}_{j+1/2} &= \Pi^R_\Delta (\mathbf{U}_{j}^\sharp,\mathbf{U}_{j}^n,\mathbf{U}_{j+1}^\sharp,\mathbf{U}_{j+1}^n),
\\
\Pi^{L,\sharp}_{j+1/2} &= \Pi^L_\Delta (\mathbf{U}_{j}^\sharp,\mathbf{U}_{j}^n,\mathbf{U}_{j+1}^\sharp,\mathbf{U}_{j+1}^n),
\end{align*}
where $ \mathbf{U}$ is the state $\begin{pmatrix}h\\h\bu\\\Pi\\z\end{pmatrix}$ and with
$$
\{gh\Delta z\}_\Delta (\mathbf{U}_L^n,\mathbf{U}_R^n)  = g\frac{h_L^n + h_R^n}{2}(z_R - z_L)
$$
$$
\Pi_\Delta (\mathbf{U}_L^{\sharp},\mathbf{U}_R^{\sharp}) = \frac{\Pi_L^{\sharp} + \Pi_R^{\sharp}}{2} - a\frac{(u_1)_R^\sharp - (u_1)_L^\sharp}{2}
$$
$$
u_\Delta (\mathbf{U}_L^\sharp,\mathbf{U}_L^n,\mathbf{U}_R^\sharp,\mathbf{U}_R^n) = \frac{(u_1)_L^\sharp + (u_1)_R^\sharp}{2} - \frac{\Pi_R^\sharp - \Pi_L^\sharp}{2a} - \frac{1}{2a}\{gh\Delta z\}_\Delta (\mathbf{U}_L^n,\mathbf{U}_R^n)
$$
$$
\Pi^L_\Delta (\mathbf{U}_L^\sharp,\mathbf{U}_L^n,\mathbf{U}_R^\sharp,\mathbf{U}_R^n) = \Pi_\Delta (\mathbf{U}_L^{\sharp},\mathbf{U}_R^{\sharp}) + \frac{1}{2}\{gh\Delta z\}_\Delta (\mathbf{U}_L^n,\mathbf{U}_R^n)
$$
$$
\Pi^R_\Delta (\mathbf{U}_L^\sharp,\mathbf{U}_L^n,\mathbf{U}_R^\sharp,\mathbf{U}_R^n) = \Pi_\Delta (\mathbf{U}_L^{\sharp},\mathbf{U}_R^{\sharp}) - \frac{1}{2}\{gh\Delta z\}_\Delta (\mathbf{U}_L^n,\mathbf{U}_R^n)
$$
If one chooses $\sharp=n$  (resp. $\sharp=n+1^-$) the system~\eqref{eq: scheme_1D_acoustic} provides a time-explicit (resp. time-implicit) discretization of the acoustic system~\eqref{eq: splitting_relaxation2}.
The approximation of the transport step is performed thanks to a standard upwind scheme for $\varphi\in\{h,hu_1,hu_2\}$
\begin{equation}\label{eq: scheme_1D_transport}
\varphi_j^{n+1} = \varphi_j^{n} - \frac{\dt}{\dx_j} \left( u_{j+1/2}^\sharp \varphi_{j+1/2}^{n+1-} - u_{j-1/2}^\sharp \varphi_{j-1/2}^{n+1-} \right)
- \frac{\dt}{\dx_j} \varphi_{j}^{n+1-} \left( u_{j+1/2}^\sharp - u_{j-1/2}^\sharp \right),
\end{equation}
where
$$
\varphi_{j+1/2}^{n+1-} =
\begin{cases}
\varphi_{j}^{n+1-},& \text{if $u_{j+1/2}^\sharp \geq 0 $,}\\
\varphi_{j+1}^{n+1-},& \text{if $u_{j+1/2}^\sharp < 0 $.}
\end{cases}
$$

Note that $\{gh\Delta z\}_\Delta$ that accounts for the gravity source term is always evaluated at time $t^n$, even for the time implicit scheme.

In the above formulas, the parameter $a$ is an approximation of the Lagrangian sound speed $hc=h\sqrt{gh}$ and must satisfy the sub-characteristic condition $a > hc$ which ensures that the relaxed system~\eqref{eq: splitting_relaxation} is a dissipative approximation of the acoustic step~\eqref{eq: splitting_acoustic} (see \cite{Bou04,chalons2008,chalons2,Despres2010-book} and the references therein). 
In order to limit the numerical diffusion we take a local approximation of the Lagrangian sound speed at every interface, given by $a_{j+1/2} = \kappa \max(h_{j}\sqrt{gh_{j}},h_{j+1}\sqrt{gh_{j+1}})$, where $\kappa>1$.

For detailed properties of the numerical scheme~\eqref{eq: scheme_1D_acoustic}-\eqref{eq: scheme_1D_transport} we refer the reader to~\cite{CKKS17}, nevertheless let us recall that: the overall discretization is conservative in the usual sense of finite volumes methods with respect to $(h,h u_1, hu_2)$.
Moreover, the scheme is also well-balanced for lake at rest conditions: if $\bu^n_j=0$ and $h^n_j + z^n_j=h^n_{j+1} + z^n_{j+1}$ for all $j\in\nbZ$, then $h^{n+1}_j=h^{n}_j$ and $\bu^{n+1}_j=\bu^{n}_j$, $j\in\nbZ$.
At last, the time-implicit scheme is stable under a condition which does not depend either on the acoustic system or the sound speed $c$, but which only depends on the transport step and its material velocity $\bu$ which is of particular interest in the low-Froude regime.

\subsection{Truncation error in the low-Froude regime}\label{ssec: scheme_truncation}
In this paragraph, we consider the dimensionless shallow-water equation and we motivate a correction of the above scheme in order to make it efficient in low-Froude regimes. The correction is similar to the one in~\cite{CGK16} for low-Mach regimes and we focus on the explicit case $\sharp=n$.

In the following we will say that the flow is in the low Froude regime if $\Fr\ll 1$ and $\px p + h \px z = \cO(\Fr^2)$.
Regarding the dimensionless equations~\eqref{eq: low_froude_swe}, we can observe that, in this regime, the variations of the discharge $hu$ remain of order $1$ as expected.

We can express the fluxes given in the previous section, using the dimensionless quantities, which leads to 
$$
u^{n}_{j+1/2} = \frac{1}{2}(u^{n}_j+u^{n}_{j+1}) - \frac{1}{2 a \Fr}\left(\Pi^{n}_{j+1}-\Pi^{n}_j + \frac{h_j^n+h_{j+1}^n}{2}(z_{j+1}-z_j)\right),
$$
$$
\Pi^{L,n}_{j+1/2} = \frac{\Pi_j^n}{\Fr^2} + \frac{1}{2\Fr^2}\left(\Pi^n_{j+1}-\Pi^n_j + \frac{h_j^n+h_{j+1}^n}{2}(z_{j+1}-z_j)\right) - \frac{a}{2\Fr}(u^n_{j+1}-u^n_j),
$$
$$
\Pi^{R,n}_{j+1/2} = \frac{\Pi_{j+1}^n}{\Fr^2} - \frac{1}{2\Fr^2}\left(\Pi^n_{j+1}-\Pi^n_j + \frac{h_j^n+h_{j+1}^n}{2}(z_{j+1}-z_j)\right) - \frac{a}{2\Fr}(u^n_{j+1}-u^n_j),
$$
if one focuses on the time-explicit scheme for the sake of simplicity.

If we compute the truncation errors in the fluxes above, using the fact that 
$$\Pi^n_{j+1}-\Pi^n_j + \frac{h_j^n+h_{j+1}^n}{2}(z_{j+1}-z_j) = \cO(\Fr^2 \dx),$$
we obtain:
$$
u^{n}_{j+1/2} = \frac{1}{2}(u^{n}_j+u^{n}_{j+1}) + \cO(\Fr \dx),
$$
$$
\Pi^{L,n}_{j+1/2} = \frac{\Pi_j^{n}}{\Fr^2} + \frac{1}{2\Fr^2}\left(\Pi^{n}_{j+1}-\Pi^{n}_j + \frac{h_j^n+h_{j+1}^n}{2}(z_{j+1}-z_j)\right) + \cO(\frac{\dx}{\Fr}),
$$
$$
\Pi^{R,n}_{j+1/2} = \frac{\Pi_{j+1}^{n}}{\Fr^2} - \frac{1}{2\Fr^2}\left(\Pi^{n}_{j+1}-\Pi^{n}_j + \frac{h_j^n+h_{j+1}^n}{2}(z_{j+1}-z_j)\right) + \cO(\frac{\dx}{\Fr}).
$$

At this stage, it is clear that the consistence errors are not uniform with respect to the Froude number in the pressure fluxes.
In order to avoid large errors in the numerical diffusion terms when the Froude tends to zero, we propose to correct the flux formula of $\Pi_\Delta (\mathbf{U}_L^{\sharp},\mathbf{U}_R^{\sharp})$ by:
$$
\Pi^\theta_\Delta (\mathbf{U}_L^{\sharp},\mathbf{U}_R^{\sharp}) = \frac{\Pi_L^{\sharp} + \Pi_R^{\sharp}}{2} - \theta a\frac{(u_1)_R^\sharp - (u_1)_L^\sharp}{2}
$$
which amounts to reduce the numerical diffusion on the pressure gradient in the low Froude regime.
Indeed, we now get
$$
\Pi^{L,n,\theta}_{j+1/2} = \frac{\Pi_j^{n}}{\Fr^2} + \frac{1}{2\Fr^2}\left(\Pi^{n}_{j+1}-\Pi^{n}_j + \frac{h_j^n+h_{j+1}^n}{2}(z_{j+1}-z_j)\right) + \cO(\frac{\theta_{j+1/2}\dx}{\Fr}),
$$
$$
\Pi^{R,n,\theta}_{j+1/2} = \frac{\Pi_{j+1}^{n}}{\Fr^2} - \frac{1}{2\Fr^2}\left(\Pi^{n}_{j+1}-\Pi^{n}_j + \frac{h_j^n+h_{j+1}^n}{2}(z_{j+1}-z_j)\right) + \cO(\frac{\theta_{j+1/2}\dx}{\Fr}),
$$
and as long as we take $\theta_{j+1/2} = \cO(\Fr)$, we recover the uniform consistency of the global scheme with respect to the Froude number.
In practice, we will set $\theta_{j+1/2}=\min \left(\frac{\lvert u^n_{j+1/2}\rvert}{\max(c_{j},c_{j+1})},1\right)$.

\subsection{The Lagrange-projection scheme on 2D unstructured meshes}
We now extend the Lagrange-projection scheme in two dimensions.
Let $\bn\in\nbR^2$ be a unit vector and $\mathbf{U}^T=(h,h\bu^T)$, we define
$$
R_\bn = 
\begin{bmatrix}
n_1 & n_2 \\
- n_2 & n_1
\end{bmatrix}
,\qquad
T_\bn \mathbf{U} = 
\begin{bmatrix}
h \\
h (R_\bn\bu)\\
\Pi\\
z 
\end{bmatrix}
.
$$

Following standard lines, we take advantage of the 
rotational invariance of the acoustic system~\eqref{eq: splitting_relaxation2} to define the two-dimensional fluxes numerical fluxes (see for example \cite{GR96,Bou04}). This leads to
\begin{subequations}\label{eq: scheme_2D_acoustic}
\begin{empheq}[left=\empheqlbrace]{align}
\label{eq: scheme_2D_acoustic_a}
\tau^{n+1-}_j &= \tau^n_j + \tau_j^n \dt \sum_{k\in \cN(j)} \sigma_{jk}\, u_{jk}^{\sharp},
\\
\label{eq: scheme_2D_acoustic_b}
\bu^{n+1-}_j &= \bu^n_j - \tau_j^n \dt \sum_{k\in \cN(j)} \sigma_{jk} \, {\Pi}_{jk}^{\sharp,\theta} \bn_{jk},
\\
\Pi^{n+1-}_j &= \Pi^n_j - \tau_j^n \dt \sum_{k\in \cN(j)} \sigma_{jk} \, (a_{jk})^2 u_{jk}^{\sharp},
\end{empheq}
\end{subequations}
where
\begin{align*}
u_{jk}^{\sharp} &= u_\Delta (T_{\bn_{jk}} \mathbf{U}_j^{\sharp}, T_{\bn_{jk}} \mathbf{U}_j^{n}, T_{\bn_{jk}} \mathbf{U}_k^{\sharp}, T_{\bn_{jk}} \mathbf{U}_k^{n}),
\\
\Pi_{jk}^{\sharp,\theta} &= \Pi_\Delta^{L,\theta} (T_{\bn_{jk}} \mathbf{U}_j^{\sharp}, T_{\bn_{jk}} \mathbf{U}_j^{n}, T_{\bn_{jk}} \mathbf{U}_k^{\sharp}, T_{\bn_{jk}} \mathbf{U}_k^{n}),
\end{align*}
that is to say
\begin{align*}
u_{jk}^{\sharp} &= \frac{1}{2} \bn_{jk}^T (\bu^\sharp_j + \bu^\sharp_k) - \frac{1}{2 a_{jk} } (\Pi^\sharp_k - \Pi^\sharp_j) - \frac{1}{2a_{jk}} \left\{gh\Delta z\right\}_{jk}^n,
\\
\Pi_{jk}^{\sharp,\theta} &= \frac{1}{2} (\Pi^\sharp_j + \Pi^\sharp_k) - \frac{a_{jk}\theta_{jk}}{2} \bn_{jk}^T( \bu^\sharp_k - \bu^\sharp_j )+\frac{1}{2}\left\{gh\Delta z\right\}_{jk}^n,
\end{align*}
with
\begin{align*}
a_{jk} &\geq \max [(h c)_j^n, (h c)_k^n],
\\
\left\{gh\Delta z\right\}_{jk}^n &= g \frac{h^n_j+h^n_k}{2} (z_k - z_j).
\end{align*}
The source term is accounted for by the terms $\Pi_{jk}^{\sharp,\theta}$ since the fluxes $\Pi_{jk}^{\sharp,\theta}\bn_{jk}$ in Equation~\eqref{eq: scheme_2D_acoustic_b} are not symmetric, indeed $\Pi_{jk}^{\sharp,\theta}\bn_{jk} \not= -\Pi_{kj}^{\sharp,\theta}\bn_{kj}$ (even if $a_{jk}=a_{kj}$ and $\theta_{jk}=\theta_{kj}$, which is the case in practice).

As far as the transport step is concerned and in order to discretize the system~\eqref{eq: splitting_transport}, we use an explicit scheme between times $t^{n+1-}$ and $t^{n+1-}+\Delta t$, where the fluxes are chosen upwind with respect to the sign of $u_{jk}^{\sharp}$.
If $\varphi \in\{ h, h\bu\}$, the scheme for the transport step reads
\begin{equation}\label{eq: scheme_2D_transport}
\varphi_j^{n+1} = \varphi_j^{n+1-} - \dt \sum_{k\in \cN(j)} \sigma_{jk} \varphi_{jk}^{n+1-} u_{jk}^{\sharp} + \dt  \varphi_j^{n+1-} \sum_{k\in \cN(j)} \sigma_{jk} u_{jk}^{\sharp},
\end{equation}
where 
\[ \varphi_{jk}^{n+1-} = 
\begin{cases}
\varphi_j^{n+1-}, &\text{if $u_{jk}^{\sharp}\geq 0$,}\\
\varphi_k^{n+1-}, &\text{if $u_{jk}^{\sharp}< 0$.}
\end{cases} 
\]

Note that one can rewrite the transport step~\eqref{eq: scheme_2D_transport} as follows
\[ \varphi_j^{n+1} =  L^{\sharp}_j \varphi_j^{n+1-} - \dt \sum_{k\in \cN(j)} \sigma_{jk} u_{jk}^{\sharp} \varphi_{jk}^{n+1-}, \]
where $L^{\sharp}_j = 1 +  \dt \sum_{k\in \cN(j)} \sigma_{jk}\, u_{jk}^{\sharp}$.
Therefore, replacing the quantities $\varphi_j^{n+1-}$ using system~\eqref{eq: scheme_2D_acoustic} gives the following update formulas which take into account the acoustic and transport steps together:
\begin{subequations}\label{eq: scheme_2D_global}
\begin{empheq}[left=\empheqlbrace]{align}
h^{n+1}_j &= h^n_j - \dt \sum_{k\in \cN(j)} \sigma_{jk} h_{jk}^{n+1-} u_{jk}^{\sharp},
\\
(h\bu)^{n+1}_j &= (h\bu)^n_j - \dt \sum_{k\in \cN(j)} \sigma_{jk}\, \left((h\bu)_{jk}^{n+1-} u_{jk}^{\sharp} + \Pi_{jk}^{\sharp,\theta} \bn_{jk}\right),
\end{empheq}
\end{subequations}
where the quantities $u_{jk}^{\sharp}$ and $\Pi_{jk}^{\sharp,\theta}$ are computed with $\bu^{\sharp}$ and $\Pi^{\sharp}$, and the quantities $h_{jk}^{n+1-}$ and $(h\bu)_{jk}^{n+1-}$ with $\tau^{n+1-}$ and $\bu^{n+1-}$ from system~\eqref{eq: scheme_2D_acoustic}.

\subsection{Stability and well-balanced properties}

Let us first notice that the scheme is conservative with respect to the water height $h$, and with respect to $h\bu$ if the topography is flat ($z=\mathrm{cste}$).
In particular, it degenerates towards the scheme proposed by~\cite{CGK16} adapted to the framework of barotropic Euler system when the bottom is flat.
Next, recall that from section~\ref{ssec: scheme_truncation}, if $\theta$ is chosen to be like $\cO(\Fr)$ when $\Fr$ goes to $0$, the truncation erro of the numerical scheme is uniform with respect to $\Fr$.
At last, assuming that the time step $\dt$ is such that the CFL conditions associated to the acoustic step
\[ \dt \underset{1\leq j \leq N}{\max}\left(\tau_j^n \underset{k\in\cN(j)}{\max} \sigma_{jk} a_{jk}\right) \leq \frac{1}{2}, \]
and to the transport step
\[ \dt \underset{1\leq j \leq N}{\max}\left(\underset{k\in\cN(j),u_{jk}^{n} < 0}{\sum} \sigma_{jk} \lvert u_{jk}^{n} \rvert \right) \leq 1, \]
hold true, the water height $h_j^n$ is positive for all $j$ and $n>0$ provided that $h_j^0$ is positive for all $j$, for the time-explicit scheme corresponding to $\sharp=n$.
Indeed, notice that $L_j^n$ turns out to be positive while the transport step correspond to a convex combination of states at time $t^{n+1-}$.

As far as the mixed implicit-explicit scheme corresponding to $\sharp=n+1-$ is concerned, the same properties hold true under the transport CFL condition
\[ \dt \underset{1\leq j \leq N}{\max}\left(\underset{k\in\cN(j),{u}_{jk}^{n+1-} < 0}{\sum} \sigma_{jk} \lvert {u}_{jk}^{n+1-} \rvert \right) \leq 1. \]
Notice that the acoustic step is implicit and therefore is free of CFL condition.

Now we show that the schemes are well-balanced.
We begin with the explicit scheme $\sharp=n$.

\begin{prop}
The full explicit scheme ($\sharp=n$) is well-balanced on 2D unstructured mesh in the sense that ($h_j^0+z_j=H=\text{cste}$ and $\bu_j^0=\mathbf{0}$) $\Longrightarrow$ ($h_j^n+z_j=\text{cste}$ and $\bu_j^n=\mathbf{0}$) $\forall j$.
\end{prop}

\begin{proof}
Assume that $h_j^0+z_j=H=\text{cste}$ and $\bu_j^0=\mathbf{0}$ $\forall j$.
We have
\begin{align*}
\left\{gh\Delta z\right\}_{jk}^0 &= g \frac{h_j^0+h_k^0}{2} \left((H-h_k^0) - (H-h_j^0)\right) = \Pi_j^0 - \Pi_k^0,
\\
u_{jk}^{0} &= - \frac{1}{2 a_{jk} } \left(\Pi_k^0-\Pi_j^0\right) - \frac{1}{2a_{jk}} \left\{gh\Delta z\right\}_{jk}^0 = 0,
\\
\Pi_{jk}^{0,\theta} &= \frac{1}{2} \left(\Pi_j^0 + \Pi_k^0\right) + \frac{1}{2} \left\{gh\Delta z\right\}_{jk}^0 = \Pi_j^0.
\end{align*}

Injecting those values in the acoustic step gives:
\begin{equation*}
\left\{
\begin{aligned}
\bu^{1-}_j &= \bu^0_j - \tau_j^0 \dt \sum_{k\in \cN(j)} \sigma_{jk}\, \Pi_j^0 \bn_{jk}
= - \tau_j^0 \Pi_j^0 \dt \sum_{k\in \cN(j)} \sigma_{jk}\, \bn_{jk} = \mathbf{0},
\\
\Pi^{1-}_j &= \Pi^0_j,
\\
\tau^{1-}_j &= \tau^0_j.
\end{aligned}
\right.
\end{equation*}

Next, since $u_{jk}^{0}=0$, $\forall j,k$, the transport step is trivial and the variables $h$ and $h\bu$ are unchanged at time $t^1$.
\end{proof}

\begin{prop}
The implicit-explict scheme ($\sharp=n+1-$) is well-balanced on 2D unstructured mesh in the sense that ($h_j^0+z_j=H=\text{cste}$ and $\bu_j^0=\mathbf{0}$) $\Longrightarrow$ ($h_j^n+z_j=\text{cste}$ and $\bu_j^n=\mathbf{0}$) $\forall j$.
\end{prop}

\begin{proof}
With the same calculus as in the explicit case, we can verify that the vector $(\bu_j^{1-},\Pi_j^{1-}) = (\mathbf{0},\Pi_j^0)$ is the only solution of the coupled system over $(\bu,\Pi)$.
Indeed we have in this case
\begin{align*}
\left\{gh\Delta z\right\}_{jk}^0 &= g \frac{h_j^0+h_k^0}{2} \left((H-h_k^0) - (H-h_j^0)\right) = \Pi_j^0 - \Pi_k^0,
\\
u_{jk}^{1-} &= - \frac{1}{2 a_{jk} } \left(\Pi_k^{1-}-\Pi_j^{1-}\right) - \frac{1}{2a_{jk}} \left\{gh\Delta z\right\}_{jk}^0 = 0,
\\
\Pi_{jk}^{1^-,\theta} &= \frac{1}{2} \left(\Pi_j^{1-} + \Pi_k^{1-}\right) + \frac{1}{2} \left\{gh\Delta z\right\}_{jk}^0 = \Pi_j^0.
\end{align*}
and therefore
\begin{equation*}
\left\{
\begin{aligned}
\bu^{1-}_j &= \bu^0_j - \tau_j^0 \dt \sum_{k\in \cN(j)} \sigma_{jk}\, \Pi_{jk}^{1^-,\theta} \bn_{jk}
= - \tau_j^0 \Pi_j^0 \dt \sum_{k\in \cN(j)} \sigma_{jk}\, \bn_{jk} = \mathbf{0},
\\
\Pi^{1-}_j &= \Pi^0_j.
\end{aligned}
\right.
\end{equation*}
Then we easily get $\tau^{1-}_j = \tau^0_j$ so that the Lagrangian step is well-balanced.
As before, we can conclude that the transport step is also well-balanced since $u_{jk}^{1-}=0$, $\forall j,k$.
\end{proof}
\section{Numerical experiments}

We present several test cases that aim at testing our scheme against classical flow configurations on unstructured meshes and also in the low Froude regime.

In the following, the EXEX scheme will refer to the full explicit scheme ($\sharp=n$) with the time step $\Delta t$ defined by
\[ \dt_{\exex} = \frac{\KCFL}{2 \max_j \left(\frac{\sum_{k\in\cN(j)} \lvert \Gamma_{jk} \rvert}{\rvert \Omega_j \rvert} \max_{k\in\cN(j)} \max(v_{jk}^{\acou},v_{jk}^{\trans}) \right)}, \]
where $\KCFL = 0.9$, $v_{jk}^{\acou} = \tau_j a_{jk}$, $a_{jk} = 1.01\max(h_j c_j,h_k c_k)$, $c_j=\sqrt{gh_j}$ and $v_{jk}^{\trans} = \lvert u_{jk}^{n} \rvert$.

The IMEX scheme will refer to the implicit-explicit scheme ($\sharp=n+1-$) with the time step $\Delta t$ defined by
\[ \dt_{\imex} = \frac{\KCFL}{2 \max_j \left(\frac{\sum_{k\in\cN(j)} \lvert \Gamma_{jk} \rvert}{\rvert \Omega_j \rvert} \max_{k\in\cN(j)} (v_{jk}^{\trans}) \right)} \]
where $v_{jk}^{\trans} = \lvert u_{jk}^{n+1-} \rvert$.
Thus, the time step of the IMEX scheme is not constrained by the acoustic waves.

Except if otherwise stated, we will always use the corrected numerical fluxes with
\[ \theta_{jk} = \min \left(\frac{\lvert u^n_{jk} \rvert}{\max\left(c_j,c_k\right)},1\right), \]
so that $\theta$ approximates a local Froude number on every edge.
On the other hand, for the sake of comparison we take $\theta = 1$ on every edge for the fluxes without correction.

\subsection{Test of the well-balanced property}\label{ssec: results_WB}
In order to test the well-balanced property of the scheme, we first consider the following lake at rest initial condition:
\begin{align*}
h(x,y,0) &= H - z(x,y),
\\
\bu(x,y,0) &= \mathbf{0},
\end{align*}
where $H = 0.5$ is constant and the topography $z$ is a smooth bump defined by
\begin{equation}\label{eq: results_WB_z}
z(x,y) = 0.3 
\times
\begin{cases}
0.5 \exp({2-\frac{0.1}{x-0.325}}), & \text{ if } 0.325 < x \leq 0.375, \\
1 - 0.5 \exp({2-\frac{0.1}{0.425-x}}), & \text{ if } 0.375 < x < 0.425, \\
1, & \text{ if } 0.425 \leq x \leq 0.575, \\
1 - 0.5 \exp({2-\frac{0.1}{x-0.575}}), & \text{ if } 0.575 < x < 0.625, \\
0.5 \exp({2-\frac{0.1}{0.675-x}}), & \text{ if } 0.625 \leq x < 0.675, \\
0 & \text{ otherwise.}
\end{cases}
\end{equation}

The physical domain $\left[0,1\right] \times \left[0,1\right]$ is discretized over a 20\,000-cell triangular mesh.
We impose Neumann boundary conditions and we observe the solution at final time $T_f = 0.1$.

For both EXEX and IMEX schemes, the errors between the numerical and the exact solution, which is also the initial stationary condition, are machine epsilon as we can observe in Table~\ref{tab: results_WB_property}.

\renewcommand{\arraystretch}{1.5}
\setlength{\tabcolsep}{0.5cm}
\begin{table}[t]
\begin{center}
\begin{tabular}{|c|c|c|}
\hline
		& $\lVert \rho - 0.5 \rVert_{\infty}$	& $\lVert \bu \rVert_{\infty}$	\\
\hline
EXEX	& $2.6 \, 10^{-16}$						& $1.3 \, 10^{-13}$				\\
\hline
IMEX 	& $2.6 \, 10^{-16}$						& $3.9 \, 10^{-8}$				\\
\hline
\end{tabular}
\end{center}
\caption{Well-balanced property. Errors for EXEX and IMEX schemes.}
\label{tab: results_WB_property}
\end{table}

\subsection{Planar dam break test problem}
We are interested in the behaviour of our schemes with regard to the propagation of a rarefaction wave and a shock wave.
We use the same triangular mesh, boundary conditions and $T_f$ value as in section~\ref{ssec: results_WB}.
The topography is also kept identical to the one given in~\eqref{eq: results_WB_z}, the velocity initialized to zero and the initial total water height $H = h+z$ is defined as follows:
\[ H(x,y,0) = 
\begin{cases}
0.5 &\text{if $x \leq 0.5$,} \\
1 &\text{otherwise.}
\end{cases} \]

In Figure~\ref{fig: results_dambreak} we present the results for both the EXEX and IMEX schemes.
We have performed a cut of the solution along the  $y=0.5$ axis and compared it with the one computed by a genuine 1D code with a 200-cell uniform grid.
We can observe that for both EXEX and IMEX the results of the 2D simulations are in agreement with the 1D results although they were computed with an unstructured mesh.
It is worth noting that the 2D simulation manages to fairly preserve the planar structure of the approximate solutions.

\begin{figure}[t]
\begin{center}
\includegraphics[width=0.7\linewidth, trim={150pt 50pt 300pt 50pt}, clip]{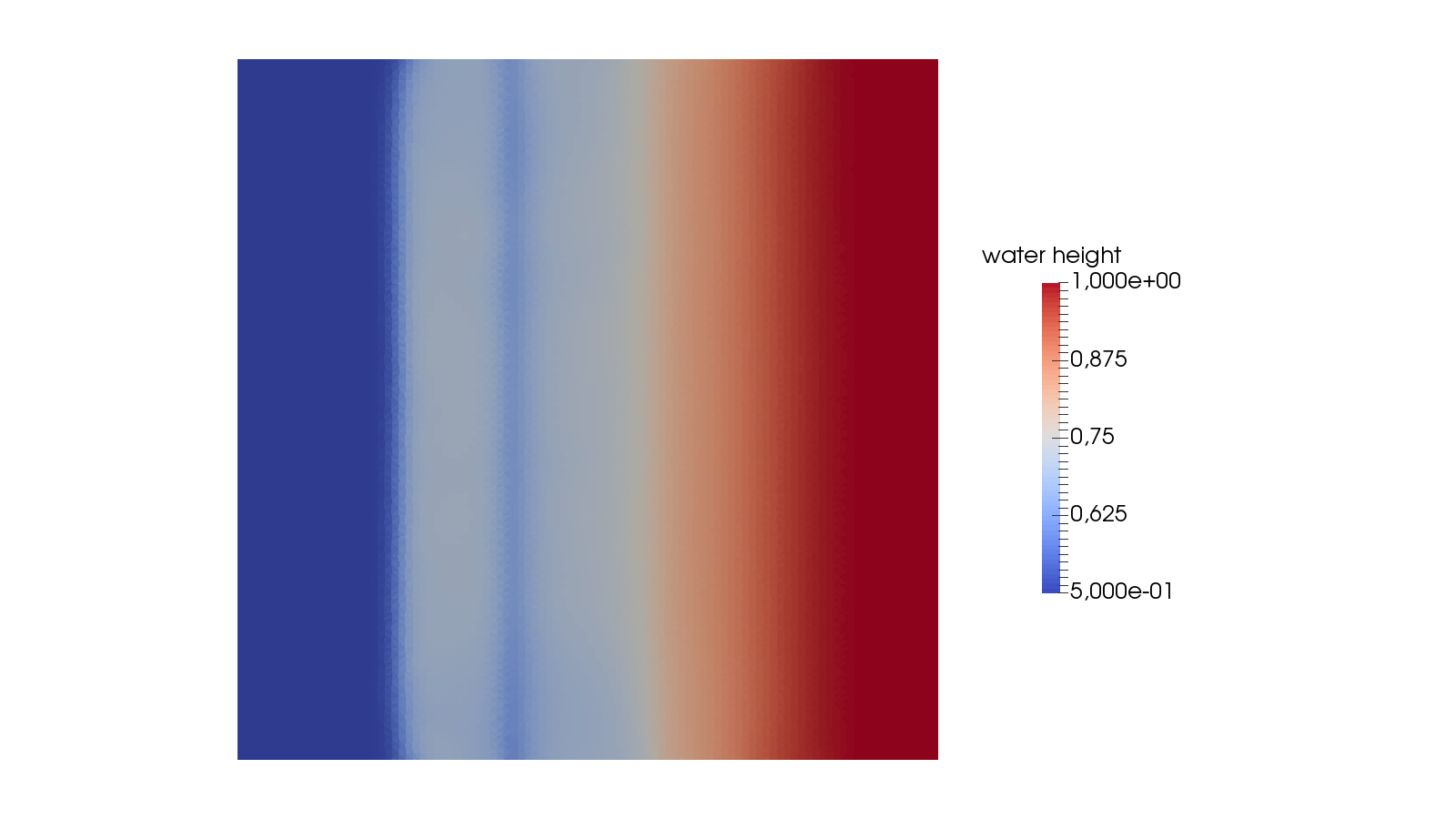}
\includegraphics[width=0.8\linewidth, trim={5pt 5pt 0pt 5pt}, clip]{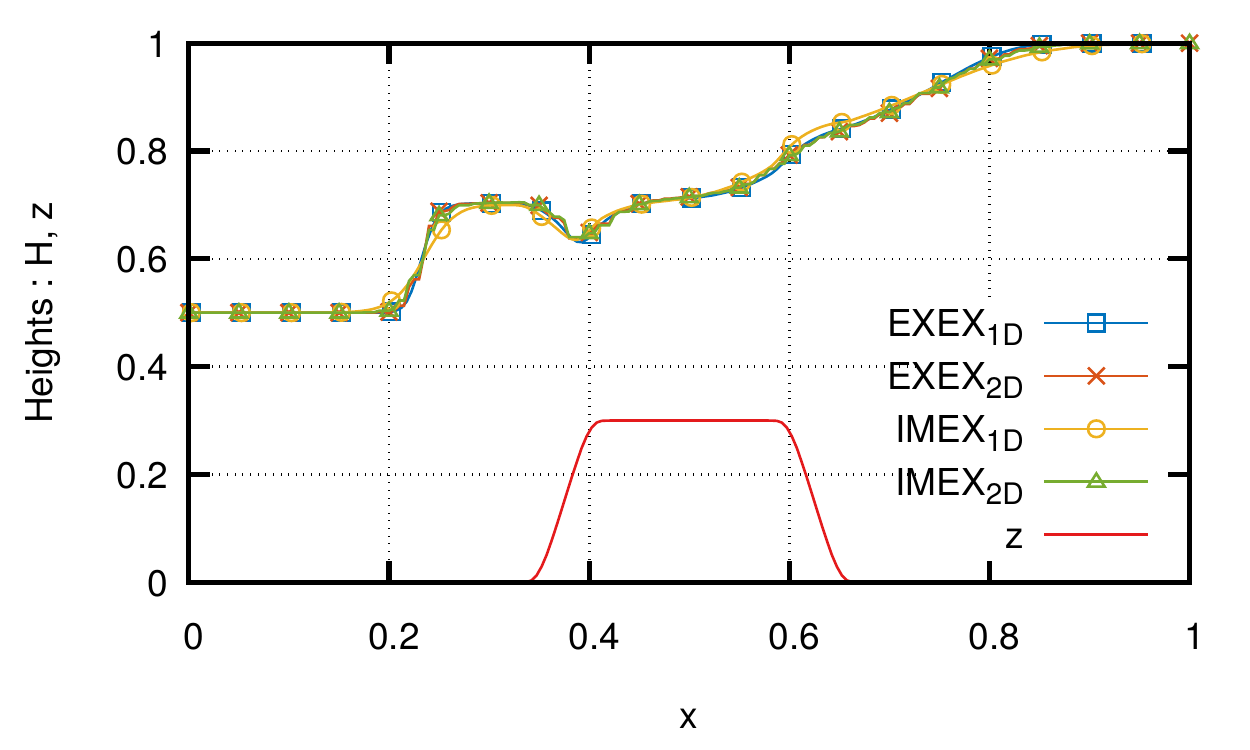}
\end{center}
\caption{Dam break test case at $T_f=0.1$: mapping of the total water height with the IMEX scheme (top) and profile of $H$ and $z$ along the $y=0.5$ axis  obtained with both the EXEX and IMEX with 2D and 1D simulations (bottom).}
\label{fig: results_dambreak}
\end{figure}

\subsection{Traveling vortex with flat bottom}\label{sec: results_vortex_flat_bottom}
In order to challenge our schemes with low Froude regimes, we consider a traveling vortex as in \cite{BALMN14}. 
The exact solution of this test is detailed in~\cite{RB09}.
For this test case we consider a flat bottom and we use a regular cartesian mesh of $160\times 160$ cells that discretizes the physical domain $\left[0,1\right] \times \left[0,1\right]$.
The boundary conditions imposed are periodic along the $x$-direction and absorbing boundaries along the $y$-direction.
The initial conditions are given by:
\begin{align*}
h(x,y,0) &= 110 +
\begin{cases}
\frac{\Gamma^2}{g\omega^2}\left(k(\omega r_c)-k(\pi)\right) &\text{if $\omega r_c \leq \pi$,} \\
0 &\text{otherwise,}
\end{cases}
\\
u(x,y,0) &= 0.6 +
\begin{cases}
\Gamma \left(1+\cos(\omega r_c)\right)\left(0.5-y\right) &\text{if $\omega r_c \leq \pi$,} \\
0 &\text{otherwise,}
\end{cases}
\\
v(x,y,0) &= 0 +
\begin{cases}
\Gamma \left(1+\cos(\omega r_c)\right)\left(x-0.5\right) &\text{if $\omega r_c \leq \pi$,} \\
0 &\text{otherwise,}
\end{cases}
\end{align*}
where
\[ r_c = \lVert \bx-(0.5,0.5)\rVert, \quad \Gamma = 15.0, \quad \omega = 4\pi, \]
and
\[ k(r) = 2\cos(r)+2r\sin(r)+\frac{1}{8}\cos(2r)+\frac{r}{4}\sin(2r)+\frac{3}{4}r^2. \]
Due to the periodic boundary conditions, the exact solution is periodic with period $T=\frac{5}{3}$ and given at any time $t>0$ by:
\begin{align*}
h(x,y,t) &= h(x-t/T,y,0),\\
u(x,y,t) &= u(x-t/T,y,0),\\
v(x,y,t) &= v(x-t/T,y,0).\\
\end{align*}

We present the results of both the EXEX and IMEX schemes, with ($\theta=\cO(\Fr)$) and without correction ($\theta=1$) using $\epsilon=0.05$.
The mapping of the velocity magnitude is displayed in Figure~\ref{fig: results_vortex_flat_bottom} and we can observe that the accuracy of the solution is really improved by the low-Froude correction.
Furthermore, the accuracy of the solution between the EXEX and the IMEX scheme with low-Froude correction is comparable whereas it took about 100 times less time steps and 10 times less CPU time computation to reach the final time with the IMEX than with the EXEX scheme as we can see in Table~\ref{tab: results_vortex_flat_bottom}.

\renewcommand{\arraystretch}{1.5}
\setlength{\tabcolsep}{0.5cm}
\begin{table}[t]
\begin{center}
\begin{tabular}{|c|c|c|}
\hline
		& Number of time steps	& CPU time	\\
\hline
EXEX	& 60264			& 1930		\\
\hline
IMEX	& 689			& 175		\\ 
\hline
\end{tabular}
\end{center}
\caption{Traveling vortex test case with flat bottom. Numbers of iterations and CPU times with low-Froude correction.}
\label{tab: results_vortex_flat_bottom}
\end{table}

\begin{figure}[t]
\begin{center}
\begin{minipage}{.84\textwidth}
\includegraphics[trim={110pt 60pt 450pt 60pt}, clip, width=0.326\textwidth]{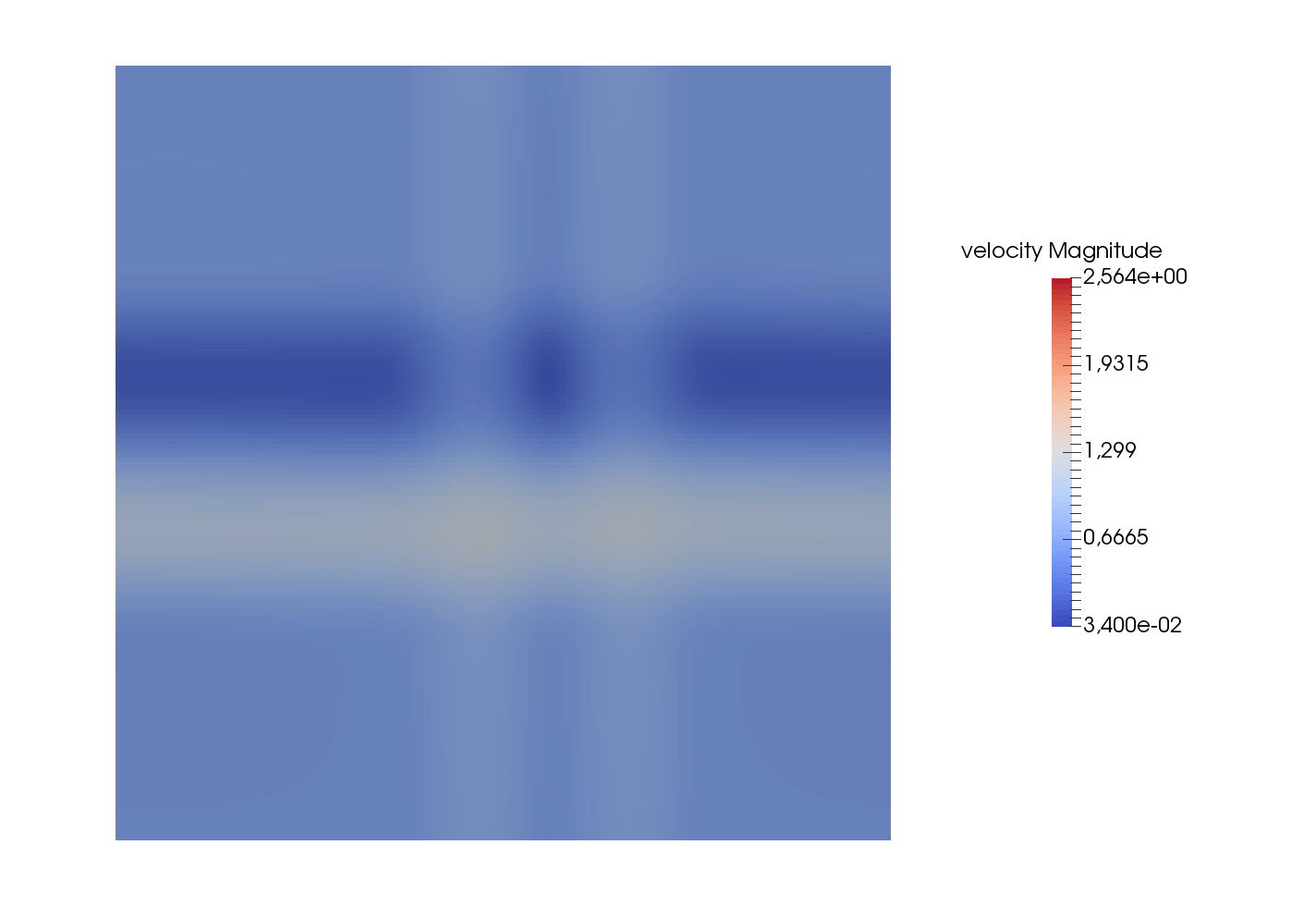}
\hfill
\includegraphics[trim={110pt 60pt 450pt 60pt}, clip, width=0.326\textwidth]{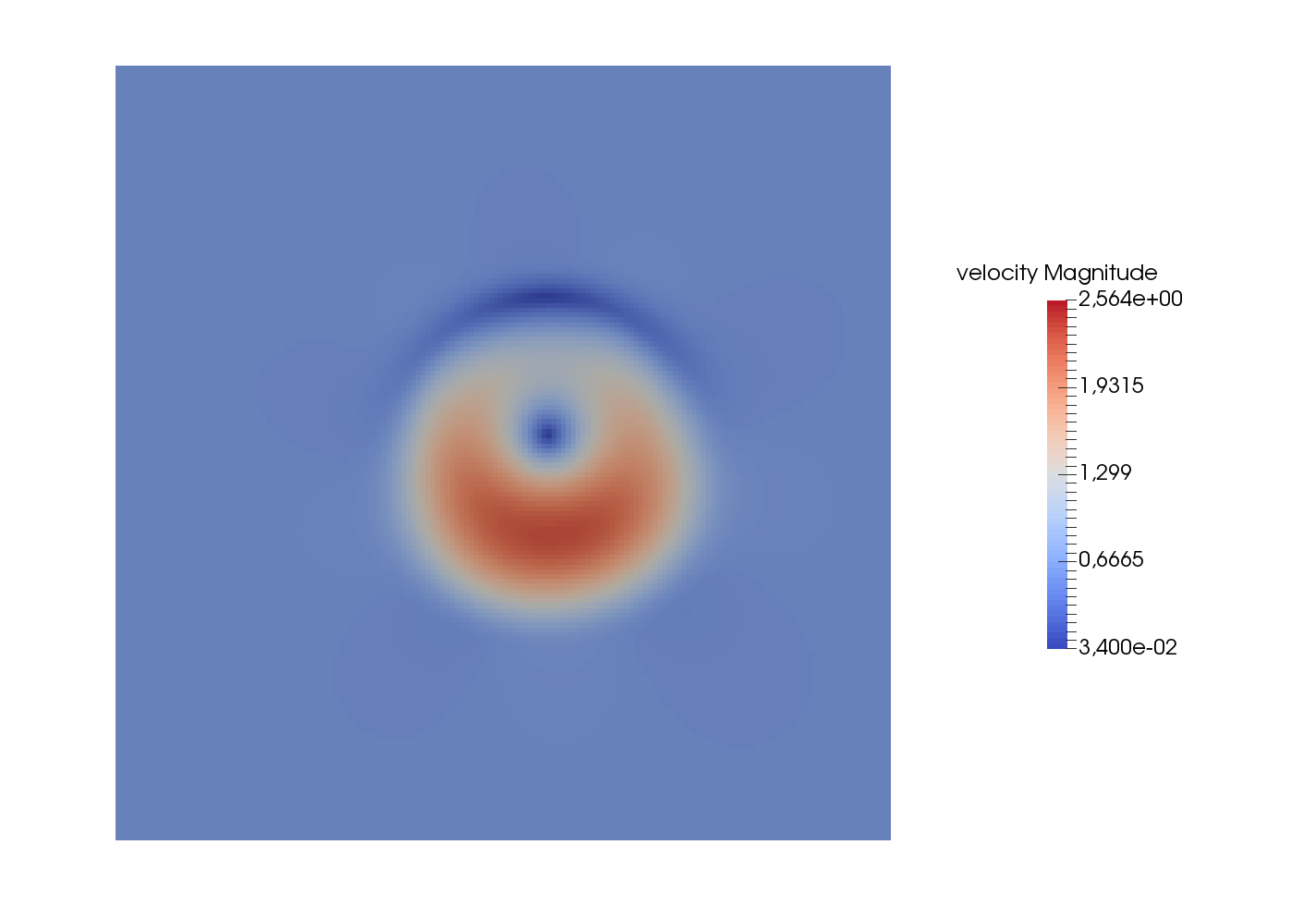}
\hfill
\includegraphics[trim={110pt 60pt 450pt 60pt}, clip, width=0.326\textwidth]{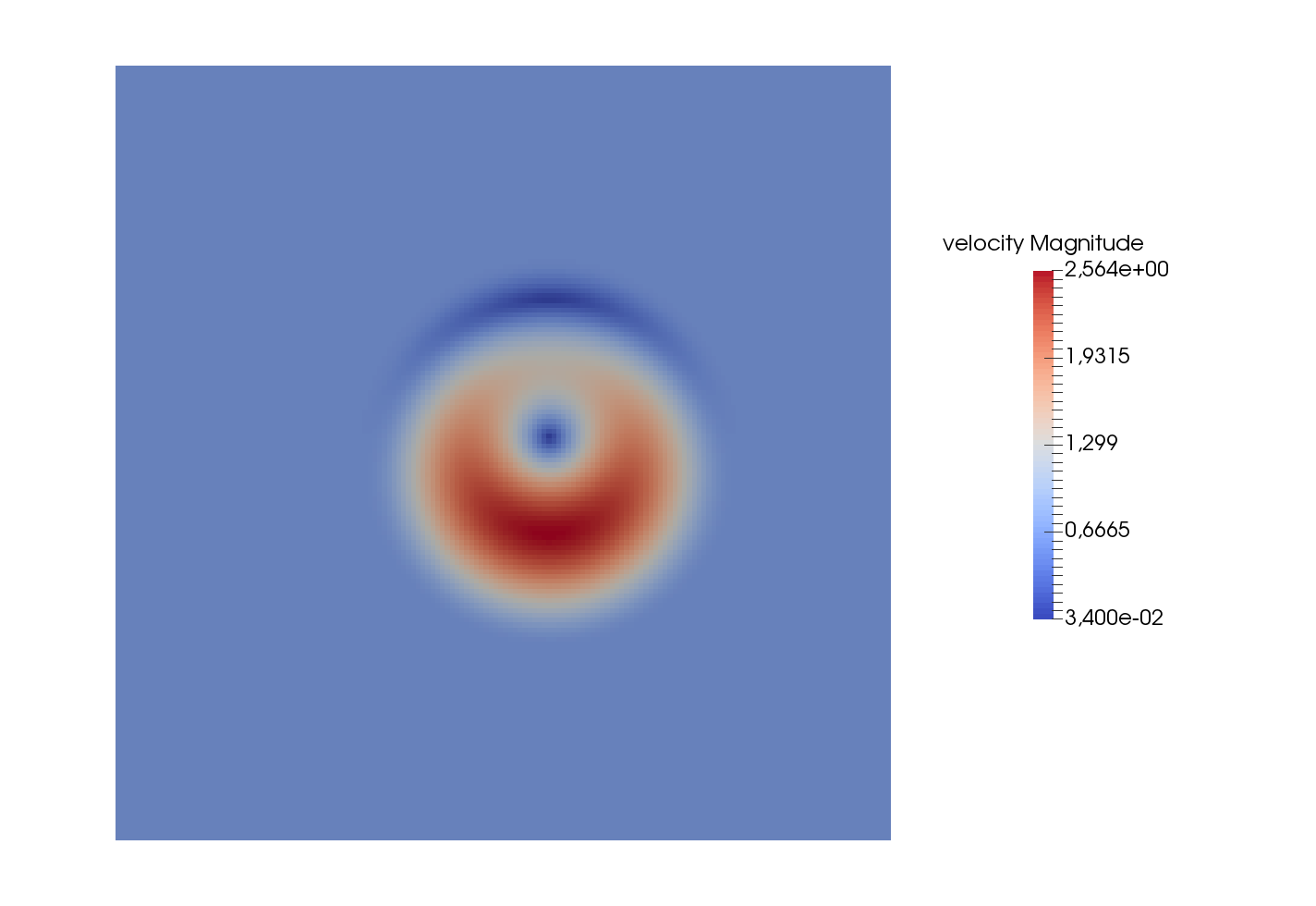}

\includegraphics[trim={110pt 60pt 450pt 60pt}, clip, width=0.326\textwidth]{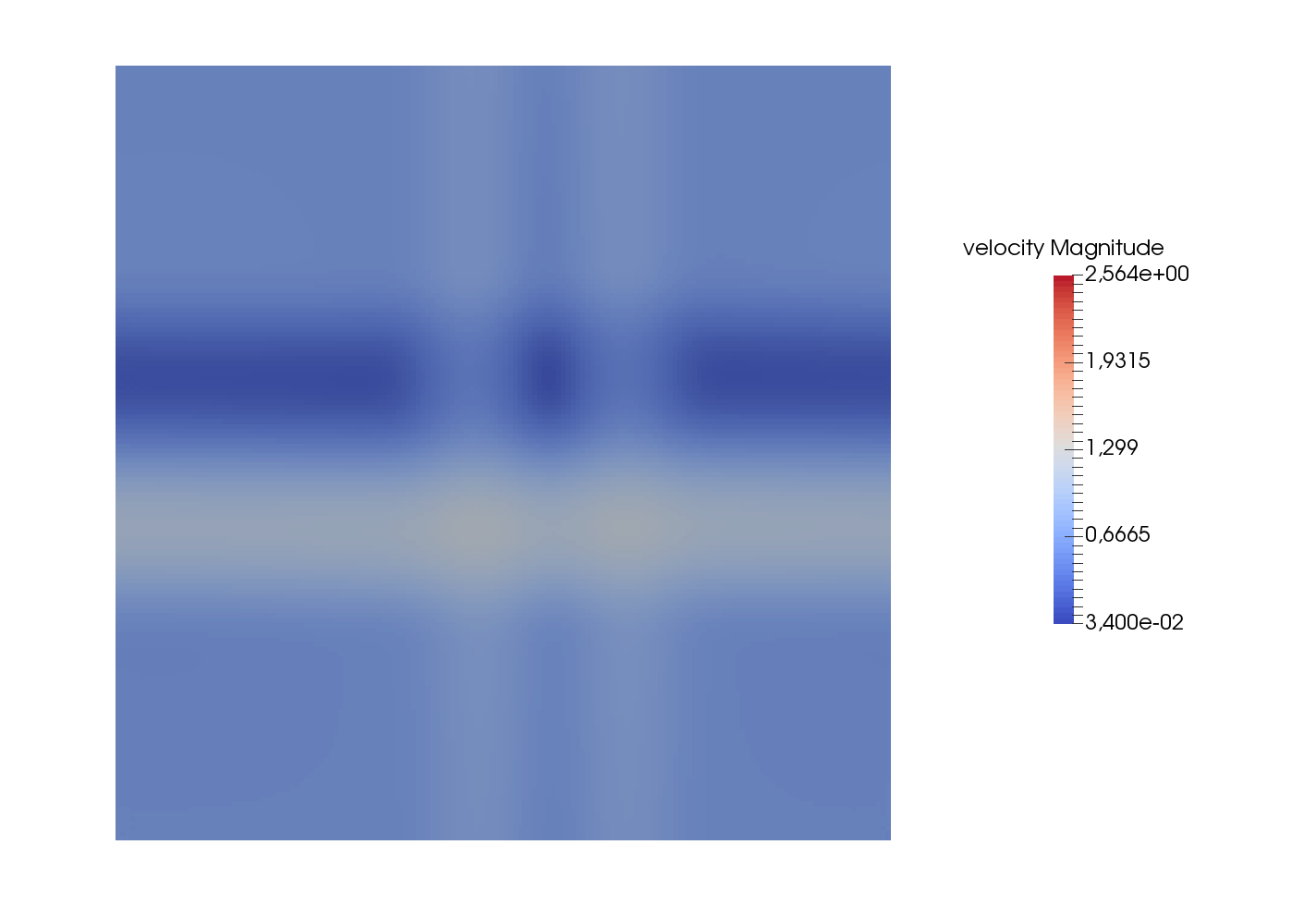}
\hfill
\includegraphics[trim={110pt 60pt 450pt 60pt}, clip, width=0.326\textwidth]{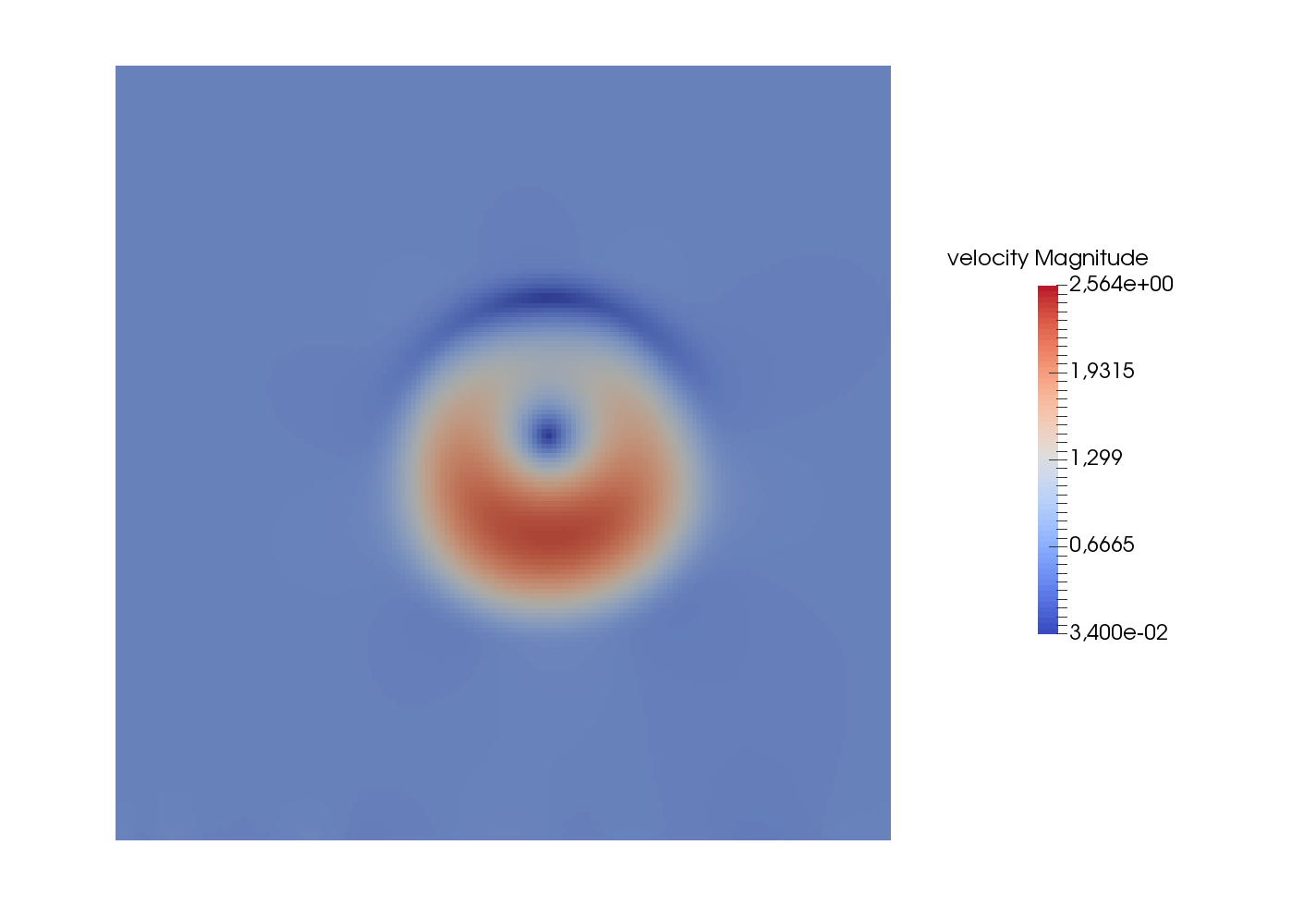}
\hfill
\includegraphics[trim={110pt 60pt 450pt 60pt}, clip, width=0.326\textwidth]{vortex_Exact.png}
\end{minipage}
\hfill
\begin{minipage}[c]{.151\textwidth}
\includegraphics[trim={1020pt 290pt 150pt 250pt}, clip, width=0.987\textwidth]{vortex_Exact.png}
\end{minipage}
\end{center}
\caption{Traveling vortex test case with flat bottom. Mapping of the velocity magnitude at $T_f=0.1$ obtained with the EXEX scheme (top) and the IMEX scheme (bottom). We used the values $\theta=1$ (left) and $\theta=\cO(\Fr)$ (center). The right column displays the exact solution.}
\label{fig: results_vortex_flat_bottom}
\end{figure}

\subsection{Traveling vortex with non-flat bottom}
We extend the physical domain of the traveling vortex test above to the rectangle $\left[0,2\right] \times \left[0,1\right]$.
The boundary conditions and initial conditions for $h$ and $u$ are the same as in section~\ref{sec: results_vortex_flat_bottom}. 
However we consider here a topography defined by $z(x,y) = 10 \exp\left(-5(x-1)^2-50(y-0.5)^2\right)$ following the idea of~\cite{BALMN14}.

We do not have exact analytical solution because of the non-flat bottom but we still can compare in figure~\ref{fig: results_vortex_non_flat_bottom} the results between EXEX and IMEX schemes, with or without low Froude correction.
Here again, the vortex structure of the flow is completely destroyed by numerical diffusion without low-Froude corrections $\theta=\cO(\Fr)$, with both schemes.
The mapping of the Froude number is not given here, but is similar to the one of the velocity magnitude, with a range of values from $1.6\cdot 10^{-3}$ to $1.1\cdot 10^{-2}$.
Finally, we can remark that the EXEX scheme took about 20 times more iterations and 15 times more CPU times than the IMEX scheme, both with low-Froude correction, as we can see in Table~\ref{tab: results_vortex_non_flat_bottom}.

\renewcommand{\arraystretch}{1.5}
\setlength{\tabcolsep}{0.5cm}
\begin{table}[t]
\begin{center}
\begin{tabular}{|c|c|c|}
\hline
		& Nb time steps	& CPU time	\\
\hline
EXEX	& 60264			& 15748		\\
\hline
IMEX	& 2733			& 921		\\ 
\hline
\end{tabular}
\end{center}
\caption{Traveling vortex test case with non-flat bottom. Numbers of iterations and CPU times with low-Froude correction.}
\label{tab: results_vortex_non_flat_bottom}
\end{table}

\begin{figure}[t]
\begin{center}
\begin{minipage}{.84\textwidth}
\includegraphics[trim={70pt 250pt 400pt 250pt}, clip, width=0.497\textwidth]{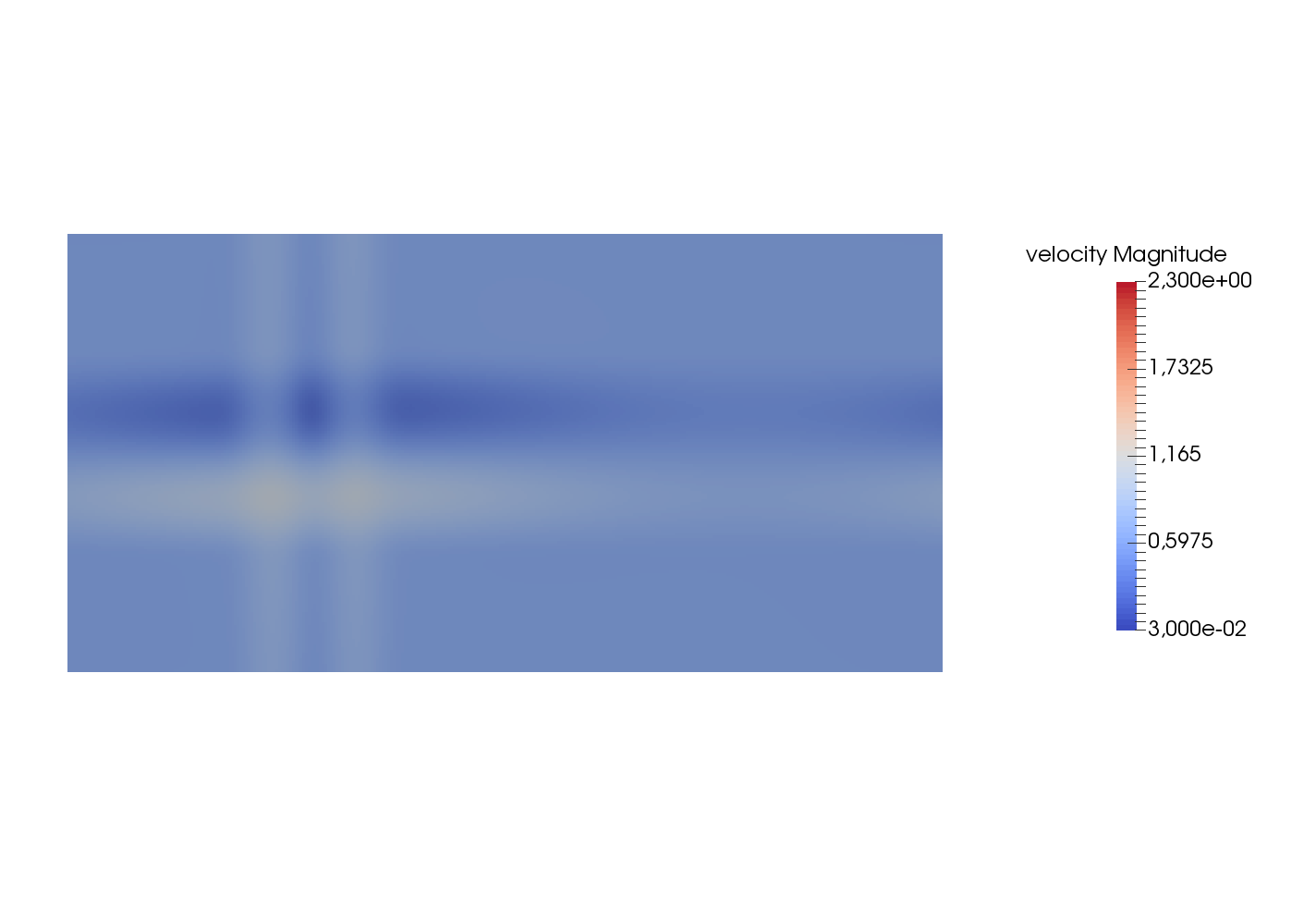}
\hfill
\includegraphics[trim={70pt 250pt 400pt 250pt}, clip, width=0.497\textwidth]{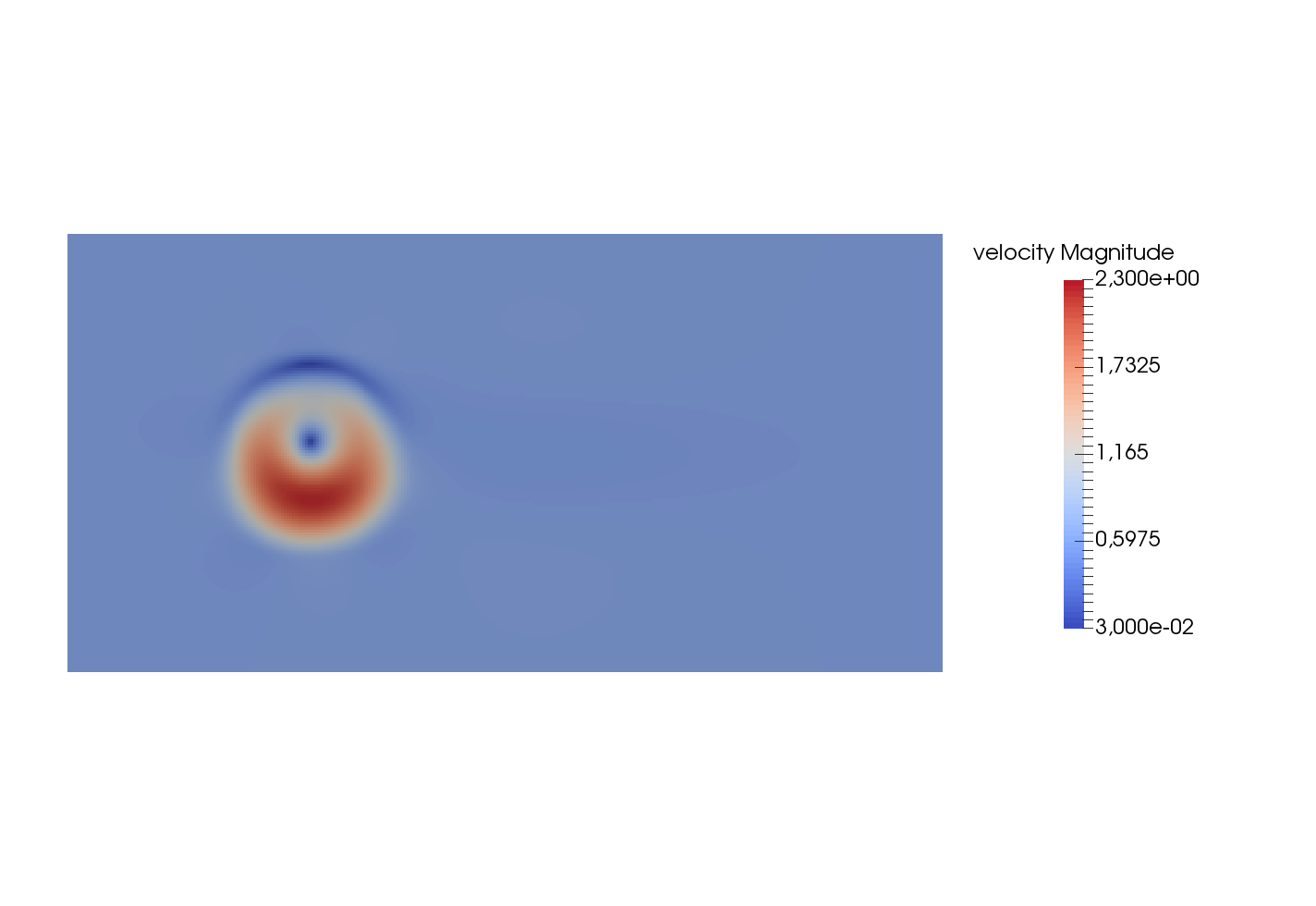}

\includegraphics[trim={70pt 250pt 400pt 250pt}, clip, width=0.497\textwidth]{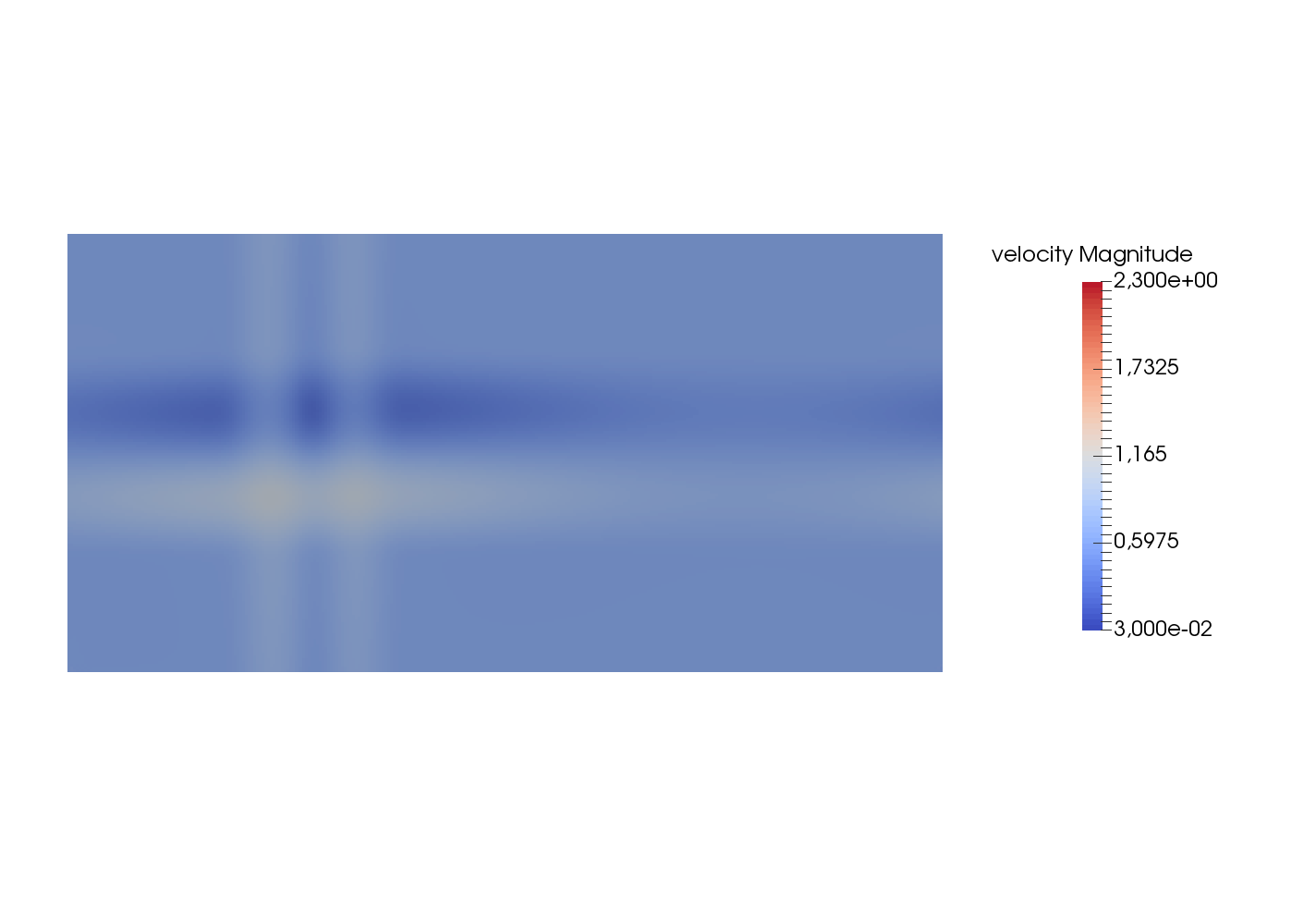}
\hfill
\includegraphics[trim={70pt 250pt 400pt 250pt}, clip, width=0.497\textwidth]{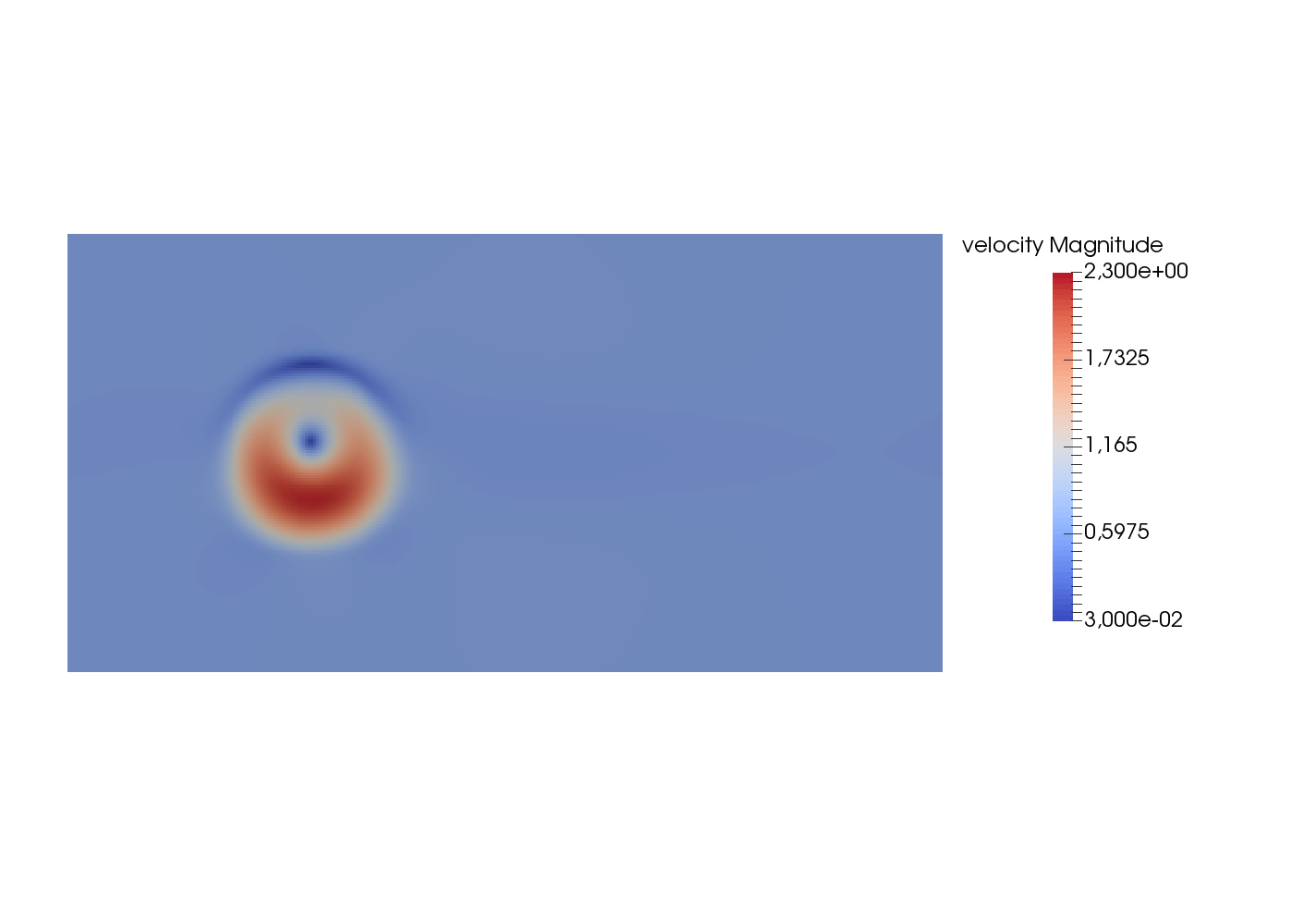}
\end{minipage}
\hfill
\begin{minipage}[c]{.151\textwidth}
\includegraphics[trim={1050pt 280pt 120pt 260pt}, clip, width=0.987\textwidth]{vortex_EXEX_with_topo.png}
\end{minipage}
\end{center}
\caption{Traveling vortex test case with non-flat bottom. Mapping of the velocity magnitude at instant $T_f=0.1$ obtained with the EXEX scheme (top) and IMEX scheme (bot) with $\theta=1$ (left) and $\theta=\cO(\Fr)$ (right).}
\label{fig: results_vortex_non_flat_bottom}
\end{figure}
\section{Conclusion}
We have proposed a large time step and well-balanced scheme for the shallow-water equations in two dimensions for unstructured meshes.
We studied the truncation of the scheme with respect to the Froude number $\Fr$ and gave a correction in accordance to the source term.
By studying the one-dimensional case, we obtained proposed of modification of the scheme that allows to obtain a uniform truncation error with respect to the Froude number for one-dimensional flows.

Moreover, we showed that the semi-implicit scheme yields good numerical results for flows from low to high Froude values since its CFL condition is based on (slow) material waves only. 

Further developments shall include extensions to high-order methods in multiple-dimensions, following for example what has already been achieved with Finite Volume or discontinuous Galerkin methods in 1D. We also intend to adapt the method presented in this work to other compressible flows models involving non-conservative terms.

\bibliography{FV2D}

\begin{thebibliography}{BALMN14}

\bibitem[BALMN14]{BALMN14}
Georgij Bispen, Koottungal~Revi Arun, M{\'a}ria
  Luk{\'a}{\v{c}}ov{\'a}-Medvid’ov{\'a}, and Sebastian Noelle.
\newblock Imex large time step finite volume methods for low {Froude} number
  shallow water flows.
\newblock {\em Communications in Computational Physics}, 16(2):307--347, 2014.

\bibitem[Bou04]{Bou04}
Fran{\c{c}}ois Bouchut.
\newblock {\em Nonlinear stability of finite Volume Methods for hyperbolic
  conservation laws: And Well-Balanced schemes for sources}.
\newblock Springer Science \& Business Media, 2004.

\bibitem[BV94]{BV94}
Alfredo Bermudez and Ma~Elena V{\'a}zquez.
\newblock Upwind methods for hyperbolic conservation laws with source terms.
\newblock {\em Computers \& Fluids}, 23(8):1049--1071, 1994.

\bibitem[CC05]{chalons2}
C.~Chalons and F.~Coquel.
\newblock Navier-stokes equations with several independent pressure laws and
  explicit predictor-corrector schemes.
\newblock {\em Numerische Mathematik}, 101(3):451--478, 2005.

\bibitem[CC08]{chalons2008}
C.~{Chalons} and J.-F. {Coulombel}.
\newblock {Relaxation approximation of the Euler equations.}
\newblock {\em {J. Math. Anal. Appl.}}, 348(2):pp.~872--893, 2008.

\bibitem[CDCdL18]{CCL18}
Manuel~J Castro~D{\'i}az, Christophe Chalons, and Tom{\'a}s~Morales de~Luna.
\newblock A fully well-balanced {Lagrange}--projection-type scheme for the
  shallow-water equations.
\newblock {\em SIAM Journal on Numerical Analysis}, 56(5):3071--3098, 2018.

\bibitem[CGK13]{CGK13}
Christophe Chalons, Mathieu Girardin, and Samuel Kokh.
\newblock Large time step and asymptotic preserving numerical schemes for the
  gas dynamics equations with source terms.
\newblock {\em SIAM Journal on Scientific Computing}, 35(6):A2874--A2902, 2013.

\bibitem[CGK14]{CGK14}
Christophe Chalons, Mathieu Girardin, and Samuel Kokh.
\newblock Operator-splitting based {AP} schemes for the {1D} and {2D} gas
  dynamics equations with stiff sources.
\newblock {\em AIMS Series on Applied Mathematics}, 8:607--614, 2014.

\bibitem[CGK16]{CGK16}
Christophe Chalons, Mathieu Girardin, and Samuel Kokh.
\newblock An all-regime {Lagrange}-projection like scheme for the gas dynamics
  equations on unstructured meshes.
\newblock {\em Communications in Computational Physics}, 20(1):188--233, 2016.

\bibitem[CKKS17]{CKKS17}
Christophe Chalons, Pierre Kestener, Samuel Kokh, and Maxime Stauffert.
\newblock A large time-step and well-balanced {Lagrange}-projection type scheme
  for the shallow water equations.
\newblock {\em Communication in Mathematical Sciences}, 15(3):765--788, 2017.

\bibitem[CNPT10]{CNPT10}
Fr{\'e}d{\'e}ric Coquel, Quang Nguyen, Marie Postel, and Quang Tran.
\newblock Entropy-satisfying relaxation method with large time-steps for
  {Euler} {IBVPs}.
\newblock {\em Mathematics of Computation}, 79(271):1493--1533, 2010.

\bibitem[Del10]{Del10}
St{\'e}phane Dellacherie.
\newblock Analysis of {Godunov} type schemes applied to the compressible
  {Euler} system at low {Mach} number.
\newblock {\em Journal of Computational Physics}, 229(4):978--1016, 2010.

\bibitem[Des10]{Despres2010-book}
B.~Despr\'es.
\newblock {\em {Lois de conservations Eul\'eriennes, Lagrangiennes et
  m\'ethodes num\'eriques}}, volume~68 of {\em Math\'ematiques et applications,
  SMAI}.
\newblock Springer, 2010.

\bibitem[Gos13]{Gos13}
Laurent Gosse.
\newblock {\em Computing qualitatively correct approximations of balance laws},
  volume~2.
\newblock Springer, 2013.

\bibitem[GR96]{GR96}
Edwige Godlewski and Pierre-Arnaud Raviart.
\newblock {\em Numerical Approximation of Hyperbolic Systems of Conservation
  Laws}, volume 118.
\newblock Springer Science \& Business Media, 1996.

\bibitem[RB09]{RB09}
Mario Ricchiuto and Andreas Bollermann.
\newblock Stabilized residual distribution for shallow water simulations.
\newblock {\em Journal of Computational Physics}, 228(4):1071--1115, 2009.

\bibitem[Zak17]{Zak17}
Hamed Zakerzadeh.
\newblock On the {Mach}-uniformity of the {Lagrange}-projection scheme.
\newblock {\em ESAIM: Mathematical Modelling and Numerical Analysis},
  51(4):1343--1366, 2017.

\end{thebibliography}

\end{document}